\documentclass[12pt,a4paper]{amsart}

\usepackage{geometry}
\usepackage{latexsym,enumerate}
\usepackage{amsmath,amsthm,amsopn,amstext,amscd,amsfonts,amssymb}

\usepackage{stackrel}
\allowdisplaybreaks
\usepackage{xcolor}
\usepackage{xfrac}

\newcommand{\R}{\mathbb{R}}
\newcommand{\N}{\mathbb{N}}

\newcommand{\car}{{\raise0pt\hbox{{\Large $\chi$}}}}
\newcommand{\sg}{{\rm \, sign}}

\newcommand{\Div}{\hbox{\rm div\,}}
\newcommand{\dis }{{\mathcal D}' }
\newcommand{\z }{{\bf z}}

\newenvironment{pf}{\noindent{\sc Proof}.\enspace}{\rule{2mm}{2mm}\medskip}
\newtheorem{Theorem}{Theorem}[section]
\newtheorem{Corollary}[Theorem]{Corollary}
\newtheorem{Definition}[Theorem]{Definition}
\newtheorem{Lemma}[Theorem]{Lemma}
\newtheorem{Proposition}[Theorem]{Proposition}

\newtheorem{Remark}[Theorem]{Remark}

\newcommand{\norma}[2]{\|#1\|_{\lower 4pt \hbox{$\scriptstyle #2$}}}
\newcommand{\ldoble}{\langle\!\langle}
\newcommand{\rdoble}{\rangle\!\rangle}

\newcommand{\intT}{\int_0^T\!\int_\Omega}
\newcommand{\intt}{\int_0^t\!\int_\Omega}
\newcommand{\ints}{\int_s^t\!\int_\Omega}
\newcommand{\intp}{\int_{\partial\Omega}}

\renewcommand{\le}{\leqslant}
\renewcommand{\ge}{\geqslant}

\def\estrella{\buildrel\ast\over\rightharpoonup }
\newcommand{\h}{{\mathcal H}^{N-1}}

\begin{document}

\title[Total Variation Flow with $L^1$-data]{The inhomogeneous Total Variation Flow with $L^1$-data}

\author[M. Latorre and  S. Segura de Le\'on]
{Marta Latorre and  Sergio Segura de Le\'on}

\address{Marta Latorre: Matem\'atica Aplicada, Ciencia e Ingenier\'ia de Materiales y Tecnolog\'ia Electr\'onica, Universidad Rey Juan Carlos
\hfill\break\indent C/Tulip\'an s/n 28933, M\'ostoles, Spain}\email{\tt marta.latorre@urjc.es}

\address{Sergio Segura de Le\'on: Departament d'An\`{a}lisi Matem\`atica,
Universitat de Val\`encia,
\hfill\break\indent Dr. Moliner 50, 46100 Burjassot, Spain}
\email  {\tt sergio.segura@uv.es }

\thanks{}

\bigskip
\begin{abstract}
This paper is devoted to the study of the Dirichlet problem for the parabolic equation driven by the $1$--Laplacian operator under minimal integrability assumptions. Specifically, we consider
\begin{equation*}
u'-\Div(Du/|D u|)=f\qquad\text{ in } (0,+\infty)\times\Omega\,,
\end{equation*}
where $\Omega\subset\R^N$ is a bounded open set with Lipschitz boundary, $u_0\in L^1(\Omega)$ is the initial datum, and $f\in L_{loc}^1(0,+\infty; L^1(\Omega))$ is the source term.

We establish the existence and uniqueness of entropy solutions in this low-regularity setting. Our approach relies on an approximation scheme and an entropy formulation adapted to the \mbox{$1$--Laplacian} structure. Additional results include comparison between solutions, further regularity  when data have higher integrability and an analysis of the long-time decay of solutions in the homogeneous case.
\end{abstract}

\maketitle

\section{Introduction}

In this work we prove the existence and uniqueness of an entropy solution to the following Dirichlet problem:
\begin{equation}\label{Pg}
\left\{\begin{array}{ll}
u'(t,x)-\Delta_1 u(t,x)=f(t,x) & \text{ in }(0,+\infty)\times\Omega\,,\\
u(t,x)=0&\text{ on } (0,+\infty)\times\partial\Omega\,,\\
u(0,x)=u_0(x)&\text{ in }\Omega\,,
\end{array}\right.
\end{equation}
where $\Omega$ is a bounded open set in $\R^N$ with $N\ge2$ having a Lipschitz boundary $\partial\Omega$. Here, $u'$ denotes the time derivative of $u$, while $\Delta_1 u = \Div(Du/|Du|)$ is the so-called 1--Laplacian (the symbol $\Div$ stands for the spacial divergence operator).  We consider a source datum $f\in L^1((0,T)\times\Omega)$, for all $T>0$, and an initial datum $u_0\in L^1(\Omega)$.

Whenever dealing with $L^1$-data, the solution cannot be expected to lie in the energy space, only its truncations can be there. In such cases, classical weak formulations of solutions are no longer applicable. Thus, to analyze $L^1$-data, the concepts of entropy and renormalized solutions were introduced. For elliptic equations similar to those driven by the $p$--Laplacian ($p>1$), a general theory was developed by  Bénilan et al. \cite{BBGGPV}, giving meaning to the gradient of the solution from the gradient of its truncations and introducing the entropy formulation in this context. 
This entropy framework was later extended to the parabolic setting. Indeed, in \cite{Pr} and \cite{AMST}  the existence and uniqueness of entropy solutions for parabolic $p$--Laplacian type equations with $L^1$-data were independently established. On the other hand, the renormalized formulation can be found in \cite{DMOP} for the elliptic case and in \cite{BM} for its parabolic adaptation (see also \cite{Po1, Po2} for related results).

One of the main challenges in working with the $1$--Laplacian operator is giving a precise meaning to the quotient $\displaystyle \frac{Du}{|Du|}$. The suitable energy space to dealing with elliptic equations driven by the $1$--Laplacian is the space of functions of bounded variation; it contains those functions whose gradient is just a Radon measure, wherewith the above quotient involves Radon measures. To overcome this difficulty, in \cite{D,ABCM} (see also \cite{ACM}), a bounded vector field $\z$ is introduced to represent the quotient $\displaystyle \frac{Du}{|Du|}$, requiring that $\|\z\|_\infty\le 1$ and the identity  $(\z,Du)=|Du|$ holds, where $(\z,Du)$ denotes the pairing supplied by Anzellotti's theory ($(\z,Du)$ is a Radon measure that generalizes the dot product, see \cite{An}). On the other hand, the boundary condition cannot be satisfied in the sense of traces and so a weak sense must be introduced, which is also based on Anzellotti's theory. Consequently, the condition  $(\z,Du)=|Du|$ must be replaced by  $(\z,DT_ku)=|DT_ku|$ for all $k>0$ in the $L^1$-framework.

These requirements were transferred to the parabolic setting in \cite{ABCM} for the homogeneous case (see, for instance, \cite{MST} for the corresponding elliptic problem). Indeed, applying the Crandall-Liggett semigroup generation theorem, the existence of an entropy solution is proven in \cite{ABCM}. By using the technique of doubling of variables, uniqueness is obtained as well. Semigroup theory also allows handle the inhomogeneous case, when $u_0\in L^2(\Omega)$ and $f\in L^2((0,T)\times \Omega)$ (see \cite{T}).  A different approach was studied in \cite{BDM2}, where solutions are parabolic minimizers of a certain related functional. Existence is achieved by passing to the limit in a sequence of minimizers of approximating strictly convex functionals. Uniqueness also holds for every convex functional. It is worth noting that both approaches yield to the same solution, at least for a regular initial datum (see \cite{KS}).

Yet another way to address the non-homogeneous problem with initial datum $u_0\in L^2(\Omega)$ was considered in \cite{SW} and \cite{LS} using an approximation procedure. In \cite{SW}, existence and uniqueness with datum $f\in L^2((0,T)\times\Omega) $ (the same setting as \cite{T}) is obtained, while the extension to $f\in L^1(0,T;L^2(\Omega))$ contained a flaw. This deficiency was corrected in the second paper through \cite[Proposition 3.8]{LS}. We point out that the procedure used, however, does not ensure  the identity  $\int_\Omega(\z(t),Du(t))=\int_\Omega|Du(t)|$ for almost all $t\in [0,T]$. It is just shown that this identity holds  {\em in mean} (see \cite[Lemma A.6 and Remark A.7]{LS}). Furthermore, although the definition of solution employed is closely related to that in \cite{ABCM},  different formulations are compared in \cite[Remark 3.7]{LS}. The last one is analogous to the weak solution of parabolic problems driven by the $p$--Laplacian (for a detailed proof that the solution found in \cite{LS} satisfies this last formulation, we refer to Proposition \ref{test-1} below). We stress that, using it, it is very easy to prove uniqueness of solutions.

In this paper, the problem we address with $L^1$-data is
\begin{equation}\label{P}
\left\{\begin{array}{ll}
u'(t,x)-\Delta_1 u(t,x)=f(t,x) & \text{ in }(0,T)\times\Omega\,,\\
u(t,x)=0&\text{ on } (0,T)\times\partial\Omega\,,\\
u(0,x)=u_0(x)&\text{ in }\Omega\,,
\end{array}\right.
\end{equation}
for every $T>0$.
Our aim is to extend the notions and results of \cite{LS} to this context. This extension is technically demanding because standard energy estimates  fail in the $L^1$-setting  (see, for instance, \cite{PPP} for a general discussion on parabolic equations driven by the $p$--Laplacian, $p>1$, and  measure data).  More precisely, we introduce a formulation of entropy solution which is similar to that considered in  \cite{Pr} or \cite{AMST} to handle $p$--Laplacian type equations, and conceptually aligned with the entropy conditions for flux-limited diffusion equations \cite{C}. Consequently, the core of our approach relies on constructing a suitable family of test functions inherently coupled with their associated vector fields, which  are provided by Proposition \ref{test-func0} below, having in mind Definition \ref{def-reg}. 

Equipped with this setting up, uniqueness can be easily proven without having to doubling variables. Regarding existence, we first extend the key result \cite[Proposition 3.8]{LS} to our framework (see Proposition 4.1 below) and get an entropy solution by constructing a suitable sequence of approximating problems. It is worth noting that the solution we found in our $L^1$-environment almost satisfies the identity $\int_\Omega (\z(t),DT_ku(t)) = \int_\Omega |DT_ku(t)|$, for almost all $t\in [0,T]$. However, we can only see it in a {\it mean sense} (see Remark \ref{igual-r}) as it occurs in the regular case. 

Besides existence and uniqueness, we establish a comparison principle and show that if the data $f$ and $u_0$ possess higher integrability --specifically, if $f(t)$ (for almost all $t>0$) and $u_0$ belong to $L^r(\Omega)$ with $1<r<2$--, then the entropy solution inherits this regularity, i.e. for almost every $t>0$ the function $u(t)$ belongs to $L^r(\Omega)$. Moreover, we  analyze the long-time behavior of solutions in the homogeneous case by showing the decay of the solution for large $t$.

The structure of the paper is as follows. In Section 2  we introduce the notation and preliminary results required to make the work self-contained. Section 3 gathers results for problems with regular data, which serve as the foundation for the general case. In Section 4 we present a key auxiliary result, while Sections 5 and 6 are devoted to proving the existence and uniqueness of entropy solutions, respectively. In Section 7 we establish a comparison result and analyze how the distance between initial data influences the distance between the corresponding solutions. Section 8 is concerned with improved regularity: we show that if the data possess higher integrability, the entropy solution reflects this fact. Finally, in Section 9, we analyze the long-time behavior of solutions to the homogeneous equation and establish a decay result when $t$ tends to $\infty$.

\section{Preliminaries}

This section collects the notation and auxiliary results used throughout the paper.

\subsection*{Notation}

Let $T>0$ be fixed and let $\Omega \subset \mathbb{R}^N$ (with $N\ge 2$) be a bounded open set with Lipschitz boundary $\partial \Omega$. We denote by  $\nu$ the outward unit normal vector to $\partial\Omega$, defined \mbox{$\mathcal{H}^{N-1}$-almost} everywhere, where $\mathcal{H}^{N-1}$ stands for the $(N-1)$-dimensional Hausdorff measure.  On the other hand, the Lebesgue measure of a measurable subset $E\subset \Omega$ is denoted by $|E|$. 

We adopt standard notation for function spaces, denoting by $L^p(\Omega)$ and $W^{1,p}(\Omega)$  the Lebesgue and Sobolev spaces, respectively. 

Given a Banach space $X$, the space $L^p(0,T;X)$ consists of strongly measurable, $p$-summable $X$--valued functions $v\colon[0,T]\to X$. The space $L_w^1(0,T;X)$ consists of weakly measurable functions, i.e.,  those  $v\in L_w^1(0,T;X)$ such that the map $t\mapsto \langle v(t),\phi\rangle$ is measurable for every $\phi\in X^*$. We recall that $L^1(0,T;X)\subset L_w^1(0,T;X)$. 

When $X$ is a function space, we  identify $v(t)\in X$ with $v(t)(x)=v(t, x)$. For further details on Banach space--valued functions, we refer to \cite[Appendix E]{Ev} and \cite{DU}.

Vector-valued Sobolev spaces are denoted by $W^{1,p}(0,T; X)$, or by $H^1(0,T;X)$ when $p=2$. Their definition and main properties can be found in \cite[Appendix]{Br}. We emphasize that these spaces satisfy the continuous embedding $W^{1,p}(0,T; X)\subset C([0,T]; X)$.

A weakly measurable function $f\colon[0,T]\to X$ is called Dunford integrable if  $\langle x^*, f(t)\rangle \in L^1(0,T)$  for every $x^*\in X^*$. If, in addition, for every measurable set $E\subset (0,T)$,  the Dunford integral $\int_E\langle x^*,f(t)\rangle\,dt$ defines an element of $X$, then $f$ is said to be Pettis integrable.

As indicated previously,  the symbol $u'$ represents the time derivative of function $u$ in the sense of distributions. The precise definition of $u'$ used in this work is given in Definition \ref{def_deriv_temporal}.

Finally, we consider the truncation operator $T_k\colon\R \to \R$, defined by
\begin{equation*}
T_k(s) = \min\{|s|,k\} \, \sg(s), \quad \text{for all } s \in \R\,,
\end{equation*}
as well as function $G_k\colon\R \to \R$
\begin{equation}\label{def_Gk}
G_k(s)=s-T_k(s)\,, \quad \text{for all } s \in \R\,.
\end{equation}
A related auxiliary function is the primitive of $T_k$, given by
\begin{equation}\label{def_Jk}
J_k(s)=\int_0^s T_k(\sigma)\,d\sigma\,, \quad \text{for all } s \in \R\,.
\end{equation}

\subsection*{Functions of Bounded Variation}

The space of functions of bounded variation, denoted by $BV(\Omega)$, constitutes the natural framework for elliptic equations involving the 1--Laplacian. A function $v\in L^1(\Omega)$ belongs to $BV(\Omega)$ if its distributional gradient $Dv$ is a Radon measure with finite total variation. We denote the total variation of $Dv$ over $\Omega$  by $\int_\Omega |Dv|$. Accordingly,  the integral of  a function $\varphi$ with respect to the measure $|Dv|$ is denoted by $\int_\Omega \varphi |Dv|$.

Recall that every $v\in BV(\Omega)$ admits a trace on the boundary satisfying $v\big|_{\partial\Omega}\in L^1(\partial \Omega)$. 

The space $BV(\Omega)$ becomes a Banach space when equipped with the norm
\begin{equation*}
\|v\|=\int_\Omega |Dv|+\intp |v|\,d\h\,.
\end{equation*}
We remark that this norm defines a functional in $BV (\Omega)$ that is lower semi-continuous with respect to the $L^1(\Omega)$-convergence.

It is worth noting that a chain rule holds for functions of bounded variation. If $u\in BV(\Omega)$ and $\varphi\colon\R\to\R$ is a Lipschitz function with a Lipschitz constant $L$, then $\varphi(u)\in BV(\Omega)$ and the inequality $|D\varphi(u)|\le L |Du|$ holds in the sense of measures (see \cite[Theorem 3.99]{AFP} and the proof of \cite[Theorem 3.96]{AFP}).

For a comprehensive study of functions of bounded variation, we refer the reader to the monographs \cite{AFP,ABM}.

We frequently use the space $BV(\Omega) \cap L^\infty(\Omega)$, equipped with the norm
\begin{equation*}
\|v\| = \max\{ \|v\|_{BV}, \|v\|_{\infty} \}\,.
\end{equation*}
This is  a   Banach space, and we denote its duality pairing with its dual $(BV(\Omega)\cap L^\infty(\Omega))^*$ and the space itself by
\begin{equation*}
\langle \zeta, v \rangle_\Omega, \quad \zeta \in (BV(\Omega)\cap L^\infty(\Omega))^*,\; v \in BV(\Omega)\cap L^\infty(\Omega)\,.
\end{equation*}
Finally, given $\xi = \zeta + f$ with $\zeta \in (BV(\Omega)\cap L^\infty(\Omega))^*$ and $f \in L^1(\Omega)$, we define the pairing
\begin{equation*}
\ldoble \xi, v \rdoble_\Omega := \langle \zeta, v \rangle_\Omega + \int_\Omega f v \, dx\qquad\text{ for all }\; v\in BV(\Omega)\cap L^\infty(\Omega)\,.
\end{equation*}

\subsection*{Anzellotti’s Theory}

To rigorously define the quotient $Du/|Du|$ and address the difficulties arising where the gradient  $Du$ vanishes in a set of nonzero measure, \cite{ABCM} introduced a vector field $\z\in L^\infty((0,T)\times\Omega;\R^N)$ to represent this term. In order to handle such vector fields, it is required a suitable version of Green's formula. We rely on the theory of $L^\infty$--divergence--measure vector fields, originally developed by Anzellotti (see \cite{An}) and extended by Chen and Frid (\cite{CF}). 

Given $\z\in L^\infty(\Omega;\R^N)$  such that $\Div\z \in L^1(\Omega)$ and  $v\in BV(\Omega)\cap L^\infty(\Omega)$, we define the pairing $(\z,Dv)$ by
\begin{equation*}
\langle\,(\z,Dv)\,,\varphi\,\rangle=-\int_\Omega v\,\varphi \,\Div\z\,dx-\int_\Omega v\,\z\cdot\nabla\varphi\,dx\,,\qquad\forall \varphi\in C_0^\infty(\Omega)\,.
\end{equation*}
This distribution $(\z,Dv)$ defines  a Radon measure satisfying the inequality $|(\z,Dv)|\le\|\z\|_\infty |Dv|$ in the sense of measures. We point out that this pairing remains well-defined under weaker regularity assumptions, for instance, when $\z\in L^\infty(\Omega;\R^N)$ with $\Div\z \in L^2(\Omega)$ and $v\in BV(\Omega)\cap L^2(\Omega)$.

Moreover, the following Green's formula holds
\begin{equation*}
\int_\Omega v\,\Div\z\,dx + \int_\Omega(\z,Dv) = \intp v\,[\z,\nu]\,d\h\,,
\end{equation*}
where $[\z,\nu]\in L^\infty(\Omega)$ stands for the weak trace on the boundary $\partial\Omega$ of the normal component of the vector field, also introduced by Anzellotti in \cite{An}.

A crucial step in our existence proof involves extending these results to cases where the integrability of $\Div\z \in L^1(\Omega)$ is a priori unknown.

A final remark regarding truncations is in order. Assume that $v\in BV(\Omega)\cap L^\infty(\Omega)$ and $\z$ is a $L^\infty$--divergence--measure vector field such that $\|\z\|_\infty\le1$. If the identity $(\z, Dv)=|Dv|$ holds in the sense of measures, it implies $(\z, DT_k(v))=|DT_k(v)|$ for every $k>0$. To see this claim, it is enough to fix $k>0$, split both measures: $|Dv|=|DT_k(v)|+|DG_k(v)|$ and $(\z, Dv)=(\z, DT_k(v))+(\z, DG_k(v))$, and write
\[
0=|Dv|-(\z, Dv)=\big[|DT_k(v)|-(\z, DT_k(v))\big]+\big[|DG_k(v)|-(\z, DG_k(v))\big]\ge0
\]
owing to the non-negativity of both bracketed terms. It is now straightforward that
$(\z, DT_k(v))=|DT_k(v)|$ and $(\z, DG_k(v))=|DG_k(v)|$ holds.

\subsection*{Time Derivative}

We now introduce the notion of time derivative employed throughout this paper.

\begin{Definition}\label{defi 1}
Let  $\Psi \in L^1(0,T; BV(\Omega))\cap L^{\infty}(0,T;
L^2(\Omega))$.

We say that $\Psi$ admits a \textbf{weak derivative} in $L_{w}^{1}(0,T;
BV(\Omega))\cap L^{1}(0,T; L^2(\Omega))$ if there exists
$\Theta\in L_{w}^{1}(0,T; BV(\Omega))\cap L^{1}(0,T; L^2(\Omega))$ such
that $\Psi(t) =\int_{0}^{t}\Theta (s) ds$, where the integral is
taken as a Pettis integral.
\end{Definition}

\begin{Definition}\label{def_deriv_temporal}
We say that a function $\xi\in L^1(0,T;BV(\Omega)\cap L^\infty(\Omega))^*+L^1((0,T)\times\Omega)$ is the \textbf{time derivative} of $u\in C([0,T];L^1(\Omega))\cap L^\infty(0,T;L^1(\Omega))$ if
\begin{equation*}
\int_0^T\ldoble \, \xi(t),\Psi(t)\,\rdoble_\Omega\,dt = -\intT u(t)\,\Phi(t)\,dx\,dt
\end{equation*}
for all $\Psi\in L^1(0,T;BV(\Omega)\cap L^\infty(\Omega))\cap L^\infty((0,T)\times\Omega)$ with compact support in $(0,T)$ and whose weak time derivative is $\Phi \in L^1_w(0,T;BV(\Omega))\cap L^1(0,T;L^\infty(\Omega))$.
\end{Definition}

\section{Regular data framework}

A key contribution of this work lies in the use of test functions equipped with an associated vector field. This approach allows the use of approximate solutions as test functions. To this end, for a fixed $T>0$, we first consider problem \eqref{P}  with regular  data $f\in L^2((0,T)\times \Omega)$ and $u_0\in W^{1,1}(\Omega)\cap L^2(\Omega)$. In this setting, the suitable concept of solution is defined as follows.

\begin{Definition}\label{def-reg}
Assume that $f\in L^2((0,T)\times \Omega)$ and $u_0\in W^{1,1}(\Omega)\cap L^2(\Omega)$.
A function $u \in L^1_w(0, T; BV(\Omega))\cap H^1(0,T;L^2(\Omega))$ (which implies  $u\in C([0,T];L^2(\Omega))$) is called a \textbf{weak solution} of \eqref{P} if $u(0) = u_0$ and there exists  a vector field $\z \in L^\infty((0,T)\times \Omega; \R^N)$ with $\|\z \|_\infty\le 1$ satisfying
\begin{enumerate}
\item[(i)]  $u'(t)=\Div \z(t) +f(t)$  in $\Omega$  in the sense of distributions,
\item[(ii)] $(\z(t) , Du(t))=|Du(t)|$ as measures in $\Omega$,
\item[(iii)] $[\z(t), \nu]\in \sg(-u(t))$ $\mathcal H^{N-1}$--a. e. on $\partial\Omega$,
\end{enumerate}
for almost all $t\in[0,T]$.

Furthermore, assuming  $f\in L^2((0,T)\times \Omega)$ for all $T>0$ and $u_0\in W^{1,1}(\Omega)\cap L^2(\Omega)$, we say that  $u\in L^1_w(0, +\infty; BV(\Omega))\cap C([0,+\infty);L^2(\Omega))$ is a global weak solution of \eqref{Pg} if it is a weak solution to \eqref{P} for every $T>0$.
\end{Definition}

\begin{Remark}\label{test}\rm
Condition (i) in Definition \ref{def-reg} implies that  any $v\in BV(\Omega)\cap L^2(\Omega)$ can be used as a test function. As a consequence, the solution itself or a truncation of it are admissible test functions. 
\end{Remark}

\begin{Remark}\rm
Since $\Div \z(t)\in L^2(\Omega)$ for almost every $t\in [0,T]$,  Anzellotti's theory applies. Then, conditions (i)--(iii) of the above definition can be combined to obtain\begin{multline}\label{simple}
\int_\Omega u'(t)\big(u(t)-\phi(t)\big)\, dx+\int_\Omega|Du(t)|+\int_{\partial\Omega}|u(t)|\, d\mathcal H^{N-1}
\\
=\int_\Omega(\z(t), D\phi(t))
-\int_{\partial\Omega}\phi(t)[\z(t), \nu]\, d\mathcal H^{N-1}+\int_\Omega f(t)\big(u(t)-\phi(t)\big)\, dx
\end{multline}
for all $\phi\in L^1_w(0, T; BV(\Omega))\cap C([0,T];L^2(\Omega))$ and almost all $t\in [0,T]$.
\end{Remark}

\begin{Proposition}\label{test-func}
For every $f\in L^2((0,T)\times \Omega)$ and every $u_0\in W^{1,1}(\Omega) \cap L^2(\Omega)$,  problem \eqref{P} admits a unique solution in the sense of Definition \ref{def-reg}.
\end{Proposition}

Proposition \ref{test-func} can be proven as in \cite[Theorem 4.1]{SW} where the solution is found as a limit of solutions to parabolic problems involving the $p$--Laplacian. Alternatively, the result can be derived  from the arguments in \cite[Theorem 1.6]{BDM2}, having in mind \cite[Theorem 5.1]{KS}. Observe that uniqueness easily follows from applying identity \eqref{simple}  to two distinct solutions, taking on account  \cite[Proposition 3.8]{LS}.

The main advantage of Proposition \ref{test-func} is that it provides a large amount of functions, each  equipped with an associated vector field, suitable for use as test functions when more general data are considered.

Moreover, since  the existence and uniqueness of solutions for $f \in L^1(0, T; L^2(\Omega))$ and $u_0 \in L^2(\Omega)$ were established in \cite[Theorems 4.1 and 5.1]{LS}, the following result shows that these solutions satisfy a formulation in which functions with an associated vector field can, indeed, be employed as test functions. This fact was announced in \cite[Remark 3.7]{LS}. The main goal of this paper is to extend that result to the $L^1$-setting.

\begin{Proposition}\label{test-1}
Assume that $f\in L^1(0,T; L^2(\Omega))$ and $u_0\in L^2(\Omega)\cap W^{1,1}(\Omega)$.
Then, the unique solution $u\in L^1_w(0,T;BV(\Omega))\cap C([0,T];L^2(\Omega))$ to problem \eqref{P} satisfies, for almost all $t\in [0,T]$,
\begin{multline}\label{fund}
\frac12 \left(\int_\Omega (u(t)-\phi(t))^2dx\right)'+\int_\Omega \phi'(t)(u(t)-\phi(t))\, dx\\
+\int_\Omega \big(\z_\phi(t), D(u(t)-\phi(t))\big)-\int_{\partial\Omega}(u(t)-\phi(t)) [\z_\phi(t), \nu]\, d\mathcal H^{N-1}\\
\le  \int_\Omega f(t)(u(t)-\phi(t))\, dx
\end{multline}
for every $\phi\in L_w^1(0,T; BV(\Omega))\cap H^{1}(0,T; L^2(\Omega))$ which has an associated vector field $\z_\phi\in L^\infty((0,T)\times\Omega; \R^N)$ such that
\begin{enumerate}
\item[(a)] $\|\z_\phi\|_\infty\le 1$,
\item[(b)] $\Div \z_\phi\in L^2((0,T)\times\Omega)$,
\item[(c)] $(\z_\phi(t), D\phi(t))=|D\phi(t)|$ for almost all $t$,
\item[(d)] $[\z_\phi(t), \nu]\in \sg(-\phi(t))$ for almost all $t$.
\end{enumerate}
\end{Proposition}
\begin{pf}
Since $f\in L^1(0,T;L^2(\Omega))$, there exists a sequence $\{ f_n\}\subset L^2((0,T)\times\Omega)$ such that $f_n\to f$ in $L^1(0,T;L^2(\Omega))$. 
 Now, consider the approximating problems
\begin{equation}\label{Pn}
\left\{\begin{array}{rcll}
\displaystyle u_n^\prime  -  \Delta_1 u_n & =& f_n(t,x)\,&\hbox{in }(0,T)\times\Omega\,,\\[2mm]
u_n &= &0\,  & \hbox{on } (0,T)\times\partial\Omega \,,\\[2mm]
u_n(0,x)&=&u_{0}(x)& \hbox{in }\Omega\,,
\end{array}\right.
\end{equation}
which have a solution $u_n\in L_w^1(0,T;BV(\Omega))\cap H^1(0,T;L^2(\Omega))$ thanks to Proposition \ref{test-func}. Moreover, it provides a vector field $\z_n\in L^\infty((0,T)\times\Omega;\R^N)$ such that $\|\z_n\|_\infty\le 1$ and 
\begin{enumerate}
\item $u_n^\prime(t) =\Div\z_n(t)+f_n(t)$ in $\dis(\Omega)$,
\item $(\z_n(t),Du_n(t)) = |Du_n(t)|$ as measures in $\Omega$,
\item $[\z_n(t),\nu]\in \sg(-u_n(t))$ $\mathcal H^{N-1}$--a. e. on $\partial\Omega$,
\end{enumerate}
holds for almost every $t\in (0,T)$.

We first claim that inequality \eqref{fund} holds for each $u_n$. To this end, choose $\phi\in L_w^1(0,T; BV(\Omega))\cap H^{1}(0,T; L^2(\Omega))$ which has an associated vector field $\z_\phi\in L^\infty((0,T)\times\Omega; \R^N)$ satisfying conditions (a)--(d) of the statement. We apply \cite[Proposition 3.8]{LS} to get
\begin{equation*}
\frac12\left(\int_\Omega (u_n(t)-\phi(t))^2dx\right)^\prime =\int_\Omega(u_n(t)-\phi(t))(u_n(t)-\phi(t))'\, dx
\end{equation*}
for almost every $t\in(0,T)$. Thus,
\begin{multline}\label{ec:01}
\frac12\left(\int_\Omega (u_n(t)-\phi(t))^2dx\right)^\prime+\int_\Omega \phi'(t)(u_n(t)-\phi(t))\, dx=
\int_\Omega u_n'(t)(u_n(t)-\phi(t))\, dx\\
=\int_\Omega \Div\z_n(t)(u_n(t)-\phi(t))\, dx+\int_\Omega f_n(t)(u_n(t)-\phi(t))\, dx\\
=-\int_\Omega (\z_n(t), D(u_n(t)-\phi(t)))+\int_{\partial\Omega}(u_n(t)-\phi(t))[\z_n(t), \nu]\, d\mathcal H^{N-1}+\int_\Omega f_n(t)(u_n(t)-\phi(t))\, dx
\end{multline}
due to Green's formula.
Hence, it remains to verify that
\begin{multline*}
-\int_\Omega (\z_n(t), D(u_n(t)-\phi(t)))+\int_{\partial\Omega}(u_n(t)-\phi(t))[\z_n(t), \nu]\, d\mathcal H^{N-1}\\
\le -\int_\Omega (\z_\phi(t), D(u_n(t)-\phi(t)))+\int_{\partial\Omega}(u_n(t)-\phi(t))[\z_\phi(t), \nu]\, d\mathcal H^{N-1}
\end{multline*}
but it is straightforward by conditions (c) and (d) of the statement. Therefore, our claim is proven.

Now, consider inequality \eqref{fund} for $u_n$, apply Green's formula, fix $0<s<t<T$ and integrate  from $s$ to $t$ (recall that $H^{1}(0,T;BV(\Omega))\subset C([0,T];BV(\Omega))$):
\begin{multline}\label{ec:02}
\frac12 \int_\Omega (u_n(t)-\phi(t))^2dx-\frac12 \int_\Omega (u_n(s)-\phi(s))^2 \,dx+\int_s^t\int_\Omega \phi'(\sigma)(u_n(\sigma)-\phi(\sigma))\, dx\, d\sigma\\
-\int_s^t\int_\Omega \Div\z_\phi(\sigma) (u_n(\sigma)-\phi(\sigma))\, d\sigma\le  \int_s^t\int_\Omega f_n(\sigma)(u_n(\sigma)-\phi(\sigma))\, dx\, d\sigma \,.
\end{multline}
Since it is proven in \cite[Theorem 4.1]{LS} that the sequence $\{u_n\}_n$ converges to $u$ (the solution to problem \eqref{P}) in $L^\infty(0,T; L^2(\Omega))$, it follows that we may let $n$ go to infinity in \eqref{ec:02} and it yields
\begin{multline*}
\frac12 \int_\Omega (u(t)-\phi(t))^2dx-\frac12 \int_\Omega (u(s)-\phi(s))^2\,dx+\int_s^t\int_\Omega \phi'(\sigma)(u(\sigma)-\phi(\sigma))\, dx\, d\sigma\\
-\int_s^t\int_\Omega \Div\z_\phi(\sigma) (u(\sigma)-\phi(\sigma))\, d\sigma \le  \int_s^t\int_\Omega f(\sigma)(u(\sigma)-\phi(\sigma))\, dx\, d\sigma\,.
\end{multline*}
Differentiating this inequality and applying Green's formula again, we are done.
\end{pf}
   
Our next concern is to get, without any intention of generality, bounded solutions to \eqref{P} that can be chosen as test functions. To begin with, we consider a Gagliardo--Nirenberg type inequality in the setting of $BV(\Omega)$.

\begin{Lemma}\label{G-N}
Let $u \in L_w^1(0, T; BV(\Omega)) \cap L^\infty(0,T; L^1(\Omega))$. 
Then $u\in L^{\frac{N+1}N}((0,T)\times\Omega)$ and the following inequality holds:
\begin{multline*}
\int_0^T\int_\Omega|G_k(u(t))|^{\frac{N+1}N}dx\, dt\\
\le C \left[\sup_{t\in(0,T)}\int_\Omega |G_k(u(t))|\, dx\right]^{\frac1N}\int_0^T\left(\int_\Omega |DG_k(u(t))|+\int_{\partial\Omega}|G_k(u(t))|\, d\mathcal H^{N-1}\right)\, dt
\end{multline*}
for a certain constant $C>0$ depending only on $N$ and $\Omega$, where $G_k(s)$ is defined in \eqref{def_Gk}.
\end{Lemma}

\begin{pf}
Applying the interpolation inequality followed by Sobolev's inequality, we obtain
\begin{multline*}
\left(\int_\Omega|G_k(u(t))|^\frac{N+1}Ndx\right)^{\frac{N}{N+1}}\le 
\left(\int_\Omega|G_k(u(t))|\, dx\right)^{\frac1{N+1}}\left(\int_\Omega|G_k(u(t))|^{\frac{N}{N-1}} dx\right)^{\frac{N-1}{N+1}}\\
\le C\left(\int_\Omega|G_k(u(t))|\, dx\right)^{\frac1{N+1}}\left(\int_\Omega|DG_k(u(t))|+\int_{\partial\Omega}|G_k(u(t))|\, d\mathcal H^{N-1}\right)^{\frac{N}{N+1}}
\end{multline*}
for almost all $t\in (0,T)$.

The result is obtained by raising both sides of the inequality to the power $\frac{N+1}N$ and integrating over $(0,T)$.
\end{pf}

\begin{Proposition}\label{test-func0}
If $f\in L^\infty((0,T)\times \Omega)$ and $u_0\in W^{1,1}(\Omega)\cap L^\infty(\Omega)$, then the unique solution to problem \eqref{P} is bounded.
\end{Proposition}

\begin{pf}
Let $u$ be the unique solution  to problem \eqref{P} with associated vector field $\z$. Now fix $k>\|u_0\|_\infty$. Taking $G_k(u)$ as test function (recall Remark \ref{test}), for almost every $0<t<T$ we obtain
\begin{multline*}
 \int_\Omega G_k(u(t ))u'(t )\,dx  + \int_\Omega |DG_k(u(t))|+  \int_{\partial\Omega}|G_k(u(t))|\, d\mathcal H^{N-1}
=    \int_\Omega f(t) G_k(u(t))\, dx \\
 \le \|f\|_\infty \int_\Omega |G_k(u(t))|\, dx.
\end{multline*}
Taking into account $u'\in L^2((0,T)\times\Omega)$, we have $G_k(u(t ))u'(t )=G_k(u(t ))G_k(u(t))'$ and so we may apply \cite[Proposition 3.8]{LS} (alternatively, we may appeal to Proposition \ref{Prop.Generalizada} below). Hence, 
\begin{equation*}
 \frac12 \left(\int_\Omega G_k(u(t ))^2\,dx\right)'  + \int_\Omega |DG_k(u(t))|+  \int_{\partial\Omega}|G_k(u(t))|\, d\mathcal H^{N-1} \le \|f\|_\infty \int_\Omega |G_k(u(t))|\, dx.
\end{equation*}
Given $0<\tau<T$, we integrate over $[0,\tau]$ to deduce
\begin{multline*}
\frac12 \int_\Omega G_k(u(\tau ))^2\,dx  +\int_0^\tau \int_\Omega |DG_k(u(t))|\, dt + \int_0^\tau \int_{\partial\Omega}|G_k(u(t))|\, d\mathcal H^{N-1}dt\\
 \le \frac12 \int_\Omega G_k(u_0)^2\,dx + \|f\|_\infty\int_0^T \int_\Omega |G_k(u(t))|\, dx\, dt
 = \|f\|_\infty\int_0^T \int_\Omega |G_k(u(t))|\, dx\, dt
\end{multline*}
due to the choice $k>\|u_0\|_\infty$.
It implies
\begin{multline}\label{ec:04}
\frac12 \sup_{\tau\in (0,T)}\int_\Omega G_k(u(\tau ))^2\,dx + \int_0^T \int_\Omega |D(G_k(u(t)))|\, dt + \int_0^T \int_{\partial\Omega}|G_k(u(t))|\, d\mathcal H^{N-1}dt\\
\le \|f\|_\infty\int_0^T  \int_\Omega |G_k(u(t))|\, dx\, dt\,.
\end{multline}
   
The next step is to apply first Lemma \ref{G-N}, then our estimate \eqref{ec:04} and finally H\"older's inequality:
\begin{multline*}
\int_0^T \int_\Omega |G_k(u(t))|^{\frac{N+1}N}\,dx\, dt\\
\le C\left[ \sup_{t\in (0,T)}\int_\Omega G_k(u(t))^2dx\right]^{\frac1{2N}}\left[\int_0^T \int_\Omega |DG_k(u(t))|\, dt + \int_0^T \int_{\partial\Omega}|G_k(u(t))|\, d\mathcal H^{N-1}dt\right]\\
\le C\left[\int_0^T  \int_\Omega |G_k(u(t))|\, dx\, dt\right]^{\frac{2N+1}{2N}}\\
\le C\left[\int_0^T \int_\Omega |G_k(u(t))|^{\frac{N+1}N}\, dt\right]^{\frac{2N+1}{2(N+1)}}|\{(t,x)\in (0,T)\times \Omega\>:\>|u(t,x)|>k\}|^{\frac{2N+1}{2N(N+1)}}
\end{multline*}
where $C$ stands for different constants that only depends on $N$ and $\Omega$.
Simplifying, it yields
\begin{equation*}
\int_0^T \int_\Omega |G_k(u(t))|^{\frac{N+1}N}\, dt\le C|\{|u|>k\}|^{\frac{2N+1}{N}}
\end{equation*}
for some constant only depending on the parameters of our problem. Since $\frac{2N+1}{N}>1$, a standard procedure (see \cite[Theorem 6.1, Chapter II]{LSU}) allows us to conclude that $u$ is bounded.   
\end{pf}

\begin{Remark}\rm
The above result can also be derived from formulation \eqref{fund}. Obviously, $T_k(u)$ cannot be taken as a test function in \eqref{fund} and an approximation with smooth functions is necessary.
Namely, for $0<\epsilon<k-\|u_0\|_\infty$, we denote by $T_k^\epsilon(s)$ a regularization of  $T_k(s)$ such that $T_k^\epsilon\in C^2(\R)$ and
\begin{equation*}
\big(T_k^\epsilon\big)'(s)=\left\{\begin{array}{ll}
1& \text{ if } |s|\le k-\epsilon\,,\\[2mm]
0 & \text{ if } |s|\ge k\,,
\end{array}\right. \quad\text{ with } \quad 0\le (T_k^\epsilon)'(s)\le 1\,.
\end{equation*}
Nevertheless, the subsequent process of letting $\epsilon$ go to 0 is cumbersome.
\end{Remark}

\section{An auxiliary result}

One of the biggest difficulties in handling the Total Variation Flow lies in the parabolic term. In the context of regular data, the evolution problem was analyzed in \cite{LS} relying on the identity
\begin{equation}\label{der-0}
\frac12\left(\int_\Omega u(t)^2dx\right)'=\int_\Omega u(t) u'(t)\, dx\,.
\end{equation}
This identity allows the use of the solution itself as a test function. However, such a property cannot be expected to hold when dealing with merely summable data. In fact, just truncations of solutions should be taken as test functions with $L^1$-data. Consequently, it is necessary to extend \eqref{der-0} to the form
\begin{equation*}
\left(\int_\Omega J_k(u(t))dx\right)'=\int_\Omega T_k(u(t)) u'(t)\, dx\,,
\end{equation*}
where $J_k(s)$ is defined in \eqref{def_Jk}. Actually, in this section we establish this identity for a general class of non-decreasing functions that includes truncations.

\begin{Proposition}\label{Prop.Generalizada}
Let $u\in H^1(0,T;L^2(\Omega))$ (so that $u\in C([0,T]; L^2(\Omega))$). Let $\varphi\colon \R\longrightarrow\R $ be a non-decreasing, bounded,  and Lipschitz-continuous function such that $\varphi(0)=0$. Define
\begin{equation*}
\psi (t)=\int_0^t \varphi(s)\,ds\,.
\end{equation*}
Then the following identity holds:
\begin{equation}\label{Prop-Gen1}
\left(\int_\Omega\psi(u(t))\,dx \right)' = \int_\Omega\varphi(u(t))\,u'(t)\,dx\,.
\end{equation}
Moreover, for any $\phi\in C_0^\infty(\Omega)$, it is satisfied
\begin{equation}\label{Prop-Gen2}
\left(\int_\Omega\phi\,\psi(u(t))\,dx \right)' = \int_\Omega\phi\,\varphi(u(t))\,u'(t)\,dx\,.
\end{equation}
\end{Proposition}

\begin{pf}
Let $\eta\in C_0^\infty(0,T)$ and take $\delta>0$ small enough so that the following calculations can be developed.

\begin{multline*}
-\intT \frac{\eta(t-\delta)-\eta(t)}{-\delta} \,\psi(u(t))\,dx\,dt
\\
= \intT  \frac{\eta(t-\delta)}{\delta}\,\psi(u(t))\,dx\,dt - \intT  \frac{\eta(t)}{\delta}\,\psi(u(t))\,dx\,dt\\
= \intT  \frac{\eta(t)}{\delta}\,\psi(u(t+\delta))\,dx\,dt - \intT  \frac{\eta(t)}{\delta}\,\psi(u(t))\,dx\,dt
\\
=\intT  \frac{\psi(u(t+\delta))-\psi(u(t))}{\delta} \,\eta(t)\,dx\,dt \,.
\end{multline*}
We define the auxiliary function
\begin{equation*}
\alpha(t,x)=\dfrac{\psi(u(t+\delta,x))-\psi(u(t,x))}{u(t+\delta,x)-u(t,x)}\quad \text{ if } u(t+\delta,x)\not=u(t,x)\,,
\end{equation*}
and $\alpha(t,x)=\varphi(u(t,x))$ if $u(t+\delta,x)=u(t,x)$. 
Thus, the previous equality can be written as
\begin{multline*}
-\intT \frac{\eta(t-\delta)-\eta(t)}{-\delta} \,\psi(u(t))\,dx\,dt=\intT  \frac{u(t+\delta)-u(t)}{\delta} \,\alpha(t)\,\eta(t)\,dx\,dt  \\
=\intT \frac{u(t+\delta)-u(t)}{\delta} \,\eta(t)\,\varphi(u(t))\,dx\,dt  +\intT \frac{u(t+\delta)-u(t)}{\delta} \,\eta(t)\,\big(\alpha(t)-\varphi(u(t))\big)\,dx\,dt \\
=A_1+A_2\,.
\end{multline*}
Consider now the function $\displaystyle \Lambda_\delta(t)=\dfrac{1}{\delta}\int_{t-\delta}^t\eta(s)\,\varphi(u(s))\,ds\in L^\infty(0,T;L^2(\Omega))$, where the integral is taken in the sense of Dunford. Notice that in \cite[Lemmas 3--4]{ABCM}, it is proven that it is actually a Pettis integral. We also remark that $\displaystyle\intT  u'(t)\,\Lambda_\delta(t)\,dx\,dt$ is well defined since $u'\in L^2((0,T)\times\Omega)$.

Therefore,
\begin{multline*}
\intT u'(t)\,\Lambda_\delta(t)\,dx\,dt = - \intT\Lambda_\delta'(t)\,u(t)\,dx\,dt \\
=- \intT \frac{\eta(t)\varphi(u(t))-\eta(t-\delta)\varphi(u(t-\delta))}{\delta}\,u(t)\,dx\,dt\\
=-\intT\frac{\eta(t)\varphi(u(t))}{\delta}\,u(t)\,dx\,dt+\intT \frac{\eta(t) \varphi(u(t))}{\delta}\,u(t+\delta)\,dx\,dt = A_1\,.
\end{multline*}

In a similar way, we define $\displaystyle\Upsilon_\delta(t)=\frac{1}{\delta}\int_{t-\delta}^t \eta(s)\big(\alpha(s)-\varphi(u(s)\big)\,ds$ which is also a Pettis integral. Let us now verify that $\Upsilon_\delta$ is bounded on $(0,T)$. Firstly, due to the Mean Value Theorem, there exists $\beta$ in the interval with endpoints $u(t+\delta)$ and $u(t)$ such that
\begin{equation*}
\psi(u(t+\delta))-\psi(u(t))=\varphi(\beta)\big(u(t+\delta)-u(t)\big)\,.
\end{equation*}
We highlight that $\beta$ depends on $x,t,\delta$ and it may not be unique; nevertheless, it allows us to write $\alpha(t)$ as $\varphi(\beta)$. Then it yields
\[
|\Upsilon_\delta(t)| \le  \frac{1}{\delta}\int_{t-\delta}^t |\eta(s)|\big|\varphi(\beta)-\varphi(u(s))\big|\,ds
\le 2\|\eta\|_{L^\infty(0,T)}\|\varphi\|_{L^\infty(\R)}\,.
\]
Moreover,
\begin{multline*}
\intT  u'(t)\,\Upsilon_\delta(t)\,dx\,dt =- \intT\Upsilon_\delta'(t)\,u(t)\,dx\,dt \\
=- \intT\frac{\eta(t)\big(\alpha(t)-\varphi(u(t))\big)-\eta(t-\delta)\big(\alpha(t-\delta)-\varphi(u(t-\delta))\big)}{\delta}\,u(t)\,dx\,dt\\
=-\intT\frac{\eta(t)\big(\alpha(t)-\varphi(u(t))\big)}{\delta}\,u(t)\,dx\,dt+\intT\frac{\eta(t)\big(\alpha(t)-\varphi(u(t))\big)}{\delta}\,u(t+\delta)\,dx\,dt = A_2\,.
\end{multline*}
Thus,
\begin{equation}\label{ec-prop}
-\intT \frac{\eta(t-\delta)-\eta(t)}{-\delta} \,\psi(u(t))\,dx\,dt=\intT u'(t)\,\Lambda_\delta(t)\,dx\,dt +\intT u'(t)\,\Upsilon_\delta(t)\,dx\,dt \,.
\end{equation}
Taking the limit as $\delta\to 0$ on the left hand side, we get
\begin{equation*}
\lim_{\delta\to0}-\intT \frac{\eta(t-\delta)-\eta(t)}{-\delta} \,\psi(u(t))\,dx\,dt =-\intT \eta'(t) \,\psi(u(t))\,dx\,dt
\end{equation*}
since  the Mean Value Theorem leads to $\left|\frac{\eta(t-\delta)-\eta(t)}{-\delta}\right|=\left|\eta'(\zeta)\right|\le\|\eta'\|_\infty$.

On the other hand, the first term on the right hand side of \eqref{ec-prop} satisfies
\begin{equation*}
\lim_{\delta\to0} \intT u'(t)\,\Lambda_\delta(t)\,dx\,dt = \intT \eta(t)\,\varphi(u(t))\,u'(t)\,dx\,dt\,;
\end{equation*}
observe that Lebesgue's Theorem can be applied due to $|\Lambda_\delta(t)|\le \|\eta\|_{L^\infty(0,T)}\|\varphi\|_{L^\infty(\R)}$.

Finally, we  show that $ \lim_{\delta\to0} \intT u'(t)\,\Upsilon_\delta(t)\,dx\,dt = 0$. Indeed, recall that we may write $\alpha(s)=\varphi(\beta)$ for some $\beta$ between $u(s+\delta)$ and $u(s)$. Hence, the monotonicity of $\varphi$ implies
\begin{equation*}
|\alpha(s)-\varphi(u(s))|=|\varphi(\beta)-\varphi(u(s))| \le |\varphi(u(s+\delta))-\varphi(u(s))|\le L |u(s+\delta)-u(s)|
\end{equation*} 
where $L$ stands for the Lipschitz constant of $\varphi$. Then, having in mind $u\in C([0,T];L^2(\Omega))$ and $u'\in L^2((0,T)\times\Omega))$, it gives
\begin{multline*}
\left|\int_\Omega \Upsilon_\delta (t)\,u'(t)\,dx\right|\le  \int_\Omega\frac{1}{\delta}\int_{t-\delta}^t |\eta(s)|\,L\,|u(s+\delta)-u(s)|\, |u'(t)|\,ds\,dx\\
\le L \|\eta\|_{L^\infty(0,T)} \sup_{t\in[0,T]}\left(\int_\Omega\big(u(t+\delta)-u(t)\big)^2\,dx\right)^{\frac12}\left(\int_\Omega|u'(t)|^2\,dx\right)^{\frac12}\in L^1(0,T)\,.
\end{multline*}
As a consequence,
\begin{equation*}
\left|\int_0^T\int_\Omega \Upsilon_\delta (t)\,u'(t)\,dx\, dt\right|\le  C\,\sup_{t\in[0,T]}\left(\int_\Omega\big(u(t+\delta)-u(t)\big)^2\,dx\right)^{\frac12} \longrightarrow 0\,.
\end{equation*}
Therefore, letting $\delta$ to $0$ in \eqref{ec-prop} we get
\begin{equation*}
-\intT \eta'(t) \,\psi(u(t))\,dx\,dt=\intT \eta(t)\,\varphi(u(t))\,u'(t)\,dx\,dt \,
\end{equation*}
for any $\eta\in C_0^\infty(0,T)$ and then \eqref{Prop-Gen1} is obtained.

Identity \eqref{Prop-Gen2} can be proven reasoning in a similar way.
\end{pf}

\begin{Remark}\rm
We point out that the above proof also holds when $u\in C([0,T];L^2(\Omega))$ and  $u'\in L^2((\epsilon,T)\times\Omega)$ for all $\epsilon>0$. This regularity occurs for solutions to problem  \eqref{P} with data $u_0\in L^2(\Omega)$ and $f\in L^2((0,T)\times\Omega)$,  as established  in \cite{T} (see also \cite{SW} and \cite{LS}).
\end{Remark}

\begin{Corollary}\label{Cont-Abs}
Assume that $f\in L^2((0,T)\times \Omega)$ and $u_0\in W^{1,1}(\Omega)\cap L^2(\Omega)$, and let $u$ denote the unique solution to problem \eqref{P}.
Let  $\varphi\colon \R\longrightarrow\R $ be a non-decreasing, bounded,  and Lipschitz-continuous function with $\varphi(0)=0$, and define
\begin{equation*}
\psi (t)=\int_0^t \varphi(s)\,ds\,.
\end{equation*}

Then the function $\displaystyle t\longmapsto \int_\Omega \psi(u(t))\,dx $
is absolutely continuous on $(0,T)$.

Moreover, the function $\displaystyle t\longmapsto \int_\Omega \phi\, \psi(u(t))\,dx $ is also absolutely continuous on $(0,T)$ for any $\phi\in C_0^\infty(\Omega)$.
\end{Corollary}

\begin{pf}
It is a consequence of
\[\left(\int_\Omega\psi(u(t))\,dx \right)' = \int_\Omega\varphi(u(t))\,u'(t)\,dx\]
since $u'\in L^2((0,T)\times\Omega)$ and $\varphi$ is bounded.
\end{pf}

\section{Entropy solution}

In this section, we introduce the notion of entropy solution we employ and prove a fundamental result that guarantees the existence of such a solution. Standard weak solution cannot be expected due to the non regularity of the data.

\begin{Definition}\label{entropy}
Let $f\in L^1((0,T)\times\Omega)$ and $u_0\in L^1(\Omega)$. We say that $u$ is an \textbf{entropy solution} to problem \eqref{P} if
\begin{itemize}
\item[(i) ]$u\in C([0,T];L^1(\Omega))$ and $T_k(u)\in L_w^1(0,T;BV(\Omega))$  for all $k>0$;
\item[(ii)] there exists $\z\in L^\infty((0,T)\times\Omega;\R^N)$ with $ \|\z\|_\infty\le1$  such that the distributional divergence $\Div \z \in L^1(0,T;BV(\Omega)\cap L^\infty(\Omega))^*$ and, for almost all $t\in[0,T]$, the pairing $(\z(t),Dv)$ is a Radon measure for all $v\in BV(\Omega)\cap L^\infty (\Omega)$;
\item[(iii)] for almost all $t\in [0,T]$, the weak trace $[\z(t),\nu]$ is well defined and it satisfies
\begin{enumerate}
\item $\|[\z(t),\nu]\|_\infty\le1$,
\item $[\z(t),\nu] \in \sg(-u(t))$ $\h$-a.e. on $\partial\Omega$,
\item the following Green's formula holds:
\begin{equation*}
\langle\,\Div\z(t),v\,\rangle_\Omega + \int_\Omega (\z(t),Dv) = \intp v\,[\z(t),\nu]\,d\h \quad\forall\,v\in BV(\Omega)\cap L^\infty(\Omega)\,;
\end{equation*}
\end{enumerate}
\item[(iv)] for every $\phi\in L^1_w(0,T;BV(\Omega))\cap H^{1}(0,T;L^2(\Omega))\cap L^\infty((0,T)\times\Omega) $ 
such that there exists $ \z_\phi\in L^\infty((0,T)\times\Omega;\R^N)$  satisfying
\begin{equation}\label{ent-for1}
\|\z_\phi\|_\infty\le 1\,,\qquad  \Div \z_\phi\in L^2((0,T)\times\Omega)\,
\end{equation}
and, for almost every  $t\in(0,T)$,
\begin{equation}\label{ent-for2}
 (\z_\phi(t),D\phi(t))=|D\phi(t)|\qquad\text{ and }\qquad [\z_\phi(t),\nu]\in\sg(-\phi(t))
\end{equation}
the following inequality holds:
\begin{multline*}
\left(\int_\Omega J_k(u(t)-\phi(t))\,dx \right)^\prime+ \int_\Omega \phi'(t)T_k(u(t)-\phi(t))\,dx  \\
\displaystyle+ \int_\Omega (\z_\phi(t),DT_k(u(t)-\phi(t)))  - \intp T_k(u(t)-\phi(t))[\z_\phi(t),\nu]\,d \h \\
\displaystyle\le \int_\Omega f(t)T_k(u(t)-\phi(t))\,dx
\end{multline*}
for almost all $t\in(0,T)$.
\end{itemize}

Given $f\in L^1((0,T)\times\Omega)$ for all $T>0$ and $u_0\in L^1(\Omega)$, we say that $u\in C([0,+\infty); L^1(\Omega))$ is a global entropy solution to \eqref{Pg} if it is a solution to \eqref{P} for every $T>0$.
\end{Definition}

\begin{Remark}\label{ext}\rm
The symbol $\Div\z(t)$ denotes the distributional divergence of the vector field $\z(t)$, which can be uniquely extended to an element of the dual space of $W_0^{1,1}(\Omega)\cap L^\infty(\Omega)$. Consequently, we may regard $\Div\z\in L^1(0,T;W_0^{1,1}(\Omega)\cap L^\infty(\Omega))^*$. By requiring that $\Div \z \in L^1(0,T;BV(\Omega)\cap L^\infty(\Omega))^*$, we select a specific extension, which is not necessarily unique. The extension of interest for our purposes is obtained as the limit (in some sense that will be determined) of the divergences of the vector fields associated to the approximating problems.
\end{Remark}

As expected, we verify that every solution in the sense of Definition \ref{def-reg} (corresponding to regular data) constitutes an entropy solution.

\begin{Proposition}\label{ent-reg}
Assume that $f\in L^2((0,T)\times \Omega)$ and $u_0\in W^{1,1}(\Omega)\cap L^2(\Omega)$. Then, the unique solution to problem \eqref{P} in the sense of Definition \ref{def-reg} is an entropy solution.
\end{Proposition}

\begin{pf}
Only condition (iv) of Definition \ref{entropy} has to be proven. Let $\phi$  satisfy the conditions set forth in (iv).

Fixed $0<t<T$, we take the test function $T_k(u(t)-\phi(t))$ in \eqref{P}:
\begin{multline}\label{ec-1}
\int_\Omega u'(t)\,T_k(u(t)-\phi(t))\,dx \\= \int_\Omega T_k(u(t)-\phi(t))\,\Div\z(t)\,dx + \int_\Omega f(t)\,T_k(u(t)-\phi(t))\,dx\,.
\end{multline}
The first integral can be written as (adding $\pm \phi'(t)$):
\begin{multline*}
\int_\Omega u'(t)\,T_k(u(t)-\phi(t))\,dx  \\
=\int_\Omega ( u(t)-\phi(t))'\,T_k(u(t)-\phi(t))\,dx  + \int_\Omega \phi'(t)\,T_k(u(t)-\phi(t))\,dx \\
=\left(\int_\Omega J_k(u(t)-\phi(t))\,dx\right)' + \int_\Omega \phi'(t)\,T_k(u(t)-\phi(t))\,dx \,,
\end{multline*}
where we have used Proposition \ref{Prop.Generalizada}.

On the other hand, the second integral of \eqref{ec-1} becomes
\begin{multline*}
\int_\Omega T_k(u(t)-\phi(t))\,\Div\z(t)\,dx \\
=\int_\Omega T_k(u(t)-\phi(t))\,\Big\{\Div\z(t)-\Div\z_\phi(t)\Big\}\,dx +\int_\Omega T_k(u(t)-\phi(t))\,\Div\z_\phi(t)\,dx \\
\le \int_\Omega T_k(u(t)-\phi(t))\,\Div\z_\phi(t)\,dx \,,
\end{multline*}
as a consequence of Green's formula since $\big(\z(t)-\z_\phi(t), DT_k(u(t)-\phi(t))\big)\ge0$ and $T_k(u(t)-\phi(t))[\z(t)-\z_\phi(t) , \nu]\le0$.
Therefore,
\begin{multline*}
\left(\int_\Omega J_k(u(t)-\phi(t))\,dx\right)' + \int_\Omega \phi'(t)\,T_k(u(t)-\phi(t))\,dx \\
-\int_\Omega T_k(u(t)-\phi(t))\,\Div\z_\phi(t)\,dx \le \int_\Omega f(t)\,T_k(u(t)-\phi(t))\,dx\,,
\end{multline*}
which shows that $u$ is also an entropy solution.
\end{pf}

With the definition established, we proceed to prove the existence of such solutions.

\begin{Theorem}\label{Teo-exist}
There exists at least one entropy solution to problem \eqref{P} for every source $f\in L^1((0,T)\times\Omega)$ and initial datum $u_0\in L^1(\Omega)$. Moreover, this solution satisfies
\begin{equation}\label{u-test}
\left(\int_\Omega J_k(u(t))\,dx\right)'+\int_\Omega |DT_k(u(t))|+\int_{\partial\Omega}|T_k(u(t))|\, d\mathcal H^{N-1}=\int_\Omega T_k(u(t)) f(t)\, dx
\end{equation}
for all $k>0$ and almost all $t\in(0,T)$.
\end{Theorem}
\begin{pf}
In order to prove the existence of entropy solutions, we follow the argument in \cite[Theorem 4.1]{LS} with slight modifications which we detail below.

We begin with the approximating problems
\begin{equation}\label{Pn-exist}
\left\{\begin{array}{ll}
u_n'-\Delta_1 u_n=f_n & \text{ in }(0,T)\times\Omega\,,\\
u_n=0&\text{ on } (0,T)\times\partial\Omega\,,\\
u_n(0,x)=u^0_n(x)&\text{ in }\Omega\,,
\end{array}\right.
\end{equation}
where we choose sequences $f_n\in L^\infty((0,T)\times\Omega)$ and $u_n^0\in L^\infty(\Omega)\cap W^{1,1}(\Omega)$ such that $f_n\to f $ in $L^1((0,T)\times\Omega)$ and $u_n^0\to u_0 $ in $L^1(\Omega)$. 

By Proposition \ref{test-func}, there exist a solution $u_n\in C([0,T];L^2(\Omega))\cap L_w^1(0,T;BV(\Omega))$ with $u_n'\in L^2((0,T)\times\Omega)$ and a vector field $\z_n\in L^\infty((0,T)\times\Omega;\R^N)$ with $\|\z_n\|_\infty\le 1$ and $\Div\z_n\in L^2((0,T)\times\Omega)$ and Proposition \ref{test-func0} guarantees that $u_n\in L^\infty((0,T)\times\Omega)$. In this case, Anzellotti's theory applies (so that a Green's formula holds) as well as,  for almost every $t \in (0,T)$, 
\begin{align}
& u_n'(t)=\Div \z_n(t)+f_n(t) \text{ in the sense of distributions},\label{cond_n0}\\
& (\z_n(t),Du_n(t))=|Du_n(t)| \text{ as measures},\label{cond_n1}\\
& [\z_n(t),\nu]\in \sg(-u_n(t))\;\h\text{-a.e. on }\partial\Omega\,.\label{cond_n2}
\end{align}

The next step is to find a function $u$ as a limit of the sequence of approximate solutions $u_n$.

Taking the test function $T_k(u_n(t)-u_m(t))$ in \eqref{cond_n0} with $f_n, u_n^0$ and also with data $f_m, u_m^0$ we get, by Proposition \ref{Prop.Generalizada},
\begin{multline*}
\left(\int_\Omega J_k (u_n(t)-u_m(t))\,dx\right)' = \int_\Omega T_k (u_n(t)-u_m(t))(u_n(t)-u_m(t))'\,dx \\
 = -\int_\Omega (\z_n(t)-\z_m(t), DT_k(u_n(t)-u_m(t)))+\intp T_k(u_n(t)-u_m(t))\Big([\z_n(t),\nu]-[\z_m(t),\nu]\Big)\,d\h\\
 +\int_\Omega (f_n(t)-f_m(t))\, T_k(u_n(t)-u_m(t))\,dx\,,
\end{multline*}
being $J_k(s)$ the function defined in \eqref{def_Jk}.

Once we have dropped non-positive terms on the right hand side and have integrated in $(0,t)$, the previous equality becomes
\begin{multline*}
\int_\Omega J_k(u_n(t) -u_m(t))\,dx \\
 \le \int_\Omega J_k(u_n^0-u_m^0)\,dx + \int_0^t\int_\Omega (f_n(s)-f_m(s))\, T_k(u_n(s)-u_m(s))\,dx\,ds \\
 \le k \int_\Omega |u_n^0-u_m^0|\,dx + k \int_0^T\int_\Omega |f_n(t)-f_m(t)|\,dx\,dt\,,
\end{multline*}
since $J_k(s)\le k|s|$ and $|T_k(s)|\le k$. Finally, dividing by $k$ and passing to the limit as $k$ goes to $0$, we obtain
\begin{equation*}
\int_\Omega|u_n(t)-u_m(t)|\,dx \le \int_\Omega |u_n^0-u_m^0|\,dx + \intT  |f_n(t)-f_m(t)|\,dx\,dt\,,
\end{equation*}
and so $\{u_n\}$ is  a Cauchy sequence in $L^\infty(0,T;L^1(\Omega))$. This fact implies that there exists $u\in C(0,T;L^1(\Omega))$ such that
\begin{align}
& u_n\longrightarrow u \quad\text{ in } L^\infty(0,T;L^1(\Omega))\,,\label{conv_1}\\
& u_n\longrightarrow u \quad\text{ in } L^1((0,T)\times\Omega)\,,\label{conv_2}\\
& u_n(t)\longrightarrow u(t) \quad\text{ in } L^1(\Omega)\,, \quad t\in (0,T)\,,\label{conv_3}\\
& u_n(t,x)\longrightarrow u(t,x) \quad\text{ a.e. in } (0,T)\times\Omega\,.\label{conv_4}
\end{align}

We now verify that the function $u$ satisfies all the conditions of Definition \ref{entropy}.

\smallskip

{\bf Condition (i).-}
We take the test function $T_k(u_n(t))$ in \eqref{cond_n0} and, since Green's formula holds, we deduce
\begin{multline*}
\int_\Omega T_k(u_n(t))\,u_n'(t)\,dx \\
= -\int_\Omega(\z_n(t),DT_k(u_n(t)))+\intp T_k(u_n(t))[\z_n(t),\nu]\,d\h+\int_\Omega f_n(t)\,T_k(u_n(t))\,dx\,.
\end{multline*}
Moreover, due to properties \eqref{cond_n1} and \eqref{cond_n2} and Proposition \ref{Prop.Generalizada}, integrating over $(0,t)$ (with $0<t<T$) it becomes
\begin{multline*}
\int_\Omega J_k(u_n(t))\,dx +\int_0^t\Big[\int_\Omega |DT_k(u_n(s))|+\intp |T_k(u_n(s))|\,d\h\Big]\,ds\\
= \int_\Omega J_k(u_n^0)\,dx+\intt f_n(s)\,T_k(u_n(s))\,dx\,ds\\
\le k\left\{\int_\Omega |u_n^0|\,dx+\intT |f_n(s)|\,dx\,ds\right\}\le k\,C\,,
\end{multline*}
which implies that 
\begin{equation*}
\int_0^T\Big[\int_\Omega |DT_k(u_n(s))|+\intp |T_k(u_n(s))|\,d\h\Big]\,ds  \le k\,C\,.
\end{equation*}
That is, every $\{T_k(u_n)\}_n$ is bounded in $L^1(0,T;BV(\Omega))$. 
It follows from the lower semicontinuity of the $BV$-norm with respect to the $L^1$-convergence and Fatou's Lemma that
\begin{multline*}
\int_0^T\Big[\int_\Omega |DT_k(u(s))|+\intp |T_k(u(s))|\,d\h\Big]\,ds 
\\
 \le \int_0^T\liminf_{n\to\infty}\Big[\int_\Omega |DT_k(u_n(s))|+\intp |T_k(u_n(s))|\,d\h\Big]\,ds 
 \\
  \le \liminf_{n\to\infty}\int_0^T\Big[\int_\Omega |DT_k(u_n(s))|+\intp |T_k(u_n(s))|\,d\h\Big]\,ds \,.
\end{multline*}
 As a consequence, $T_k(u(t))\in BV(\Omega)$ for almost every $t\in (0,T)$ and, by \cite[Lemma 5.19]{ACM}, the function $t\mapsto \int_\Omega |DT_k(u(t))|$ is measurable and $T_k(u)\in L_w^1(0,T;BV(\Omega))$.

\smallskip

{\bf Condition (ii).-}
With regard to the vector field $\z$, its existence follows directly from the uniform bound $\|\z_n\|_\infty\le 1$. Indeed, passing to a subsequence if necessary, there exists $\z\in L^\infty((0,T)\times\Omega;\R^N)$ such that $\|\z\|_\infty\le 1$ and $\z_n\estrella\z$ in $L^\infty((0,T)\times\Omega;\R^N)$.

To show that the sequence $\{\Div\z_n\}$ is bounded in the space $L^1(0,T;BV(\Omega)\cap L^\infty(\Omega))^*$ we follow the argument in \cite[Theorem 4.1, Step 6]{LS} but with $v\in L^1(0,T;BV(\Omega)\cap L^\infty(\Omega))$. Hence, we find a subnet such that $\Div\z_\alpha\estrella\Div \z$ in $L^1(0,T;BV(\Omega)\cap L^\infty(\Omega))^*$ (here $\Div \z$ denotes a suitable extension of the distributional gradient of $\z$).

Moreover, we also define
\begin{equation*}
\xi = \Div\z + f\; \in\; L^1(0,T;BV(\Omega)\cap L^\infty(\Omega))^*+ L^1((0,T)\times\Omega)\,,
\end{equation*}
and working as in \cite[Theorem 4.1, Step 7]{LS} but with test functions $v\in L^1(0,T;BV(\Omega)\cap L^\infty(\Omega))\cap L^\infty((0,T)\times\Omega)$, we prove that
\begin{equation*}
u_\alpha'\estrella \xi \;\in \; L^1(0,T;BV(\Omega)\cap L^\infty(\Omega))^*+L^1((0,T)\times\Omega)\,.
\end{equation*}

On the other hand, \cite[Theorem 4.1, Step 8]{LS} implies that, for almost every $t \in (0,T)$, the equation holds in the distributional sense:
\begin{equation*}
\left(\int_\Omega \omega\,u(t)\,dx\right)' =\ldoble \xi, \omega \rdoble_\Omega= \int_\Omega \z(t)\cdot\nabla \omega\,dx + \int_\Omega f(t)\,\omega\,dx \quad \forall \omega\in
C_0^\infty(\Omega)\,,
\end{equation*}
while \cite[Theorem 4.1, Step 9]{LS} shows that $(\z(t),Dv)$  is a Radon measure for all $v\in BV(\Omega)\cap L^\infty(\Omega)$. This completes the verification of condition (ii) in Definition \ref{entropy}. 

\smallskip

{\bf Condition (iii).-}
Arguing as in \cite[Theorem 4.1, Steps 10-12]{LS} we can define $[\z(t),\nu]\in L^\infty(\partial\Omega)$ with $\|[\z(t),\nu]\|_\infty\le1$ and such that a Green's formula holds:
\begin{equation*}
\langle\,\Div\z(t),v\,\rangle_\Omega+\int_\Omega \z(t)\cdot\nabla v \,dx =\intp v[\z(t),\nu]\,d\h
\end{equation*}
for all $v\in BV(\Omega)\cap L^\infty(\Omega) $. 
Furthermore, we may show that $\xi$ is the time derivative of $u$ in the sense of Definition \ref{def_deriv_temporal} following the argument of \cite[Theorem 4.1, Step 13]{LS}.

To conclude the verification of condition (iii) in Definition \ref{entropy}, we must prove that the boundary condition (iii) \textit{(2)} holds. To this end, set $k>0$ and let $\eta\in C_0^\infty(0,T)$ be non-negative. For every $\delta>0$ small enough,  define
\begin{equation*}
\Psi_\delta(t)=\frac1{\delta}\int_{t-\delta}^t\eta(s) T_k(u(s))\, ds\,,
\end{equation*}
that is a Pettis integral as a consequence of \cite[Lemmas 3-4]{ABCM}.
On the other hand, observe that the integrals
\begin{equation*}
 \int_0^T \langle\, \Div \z(t),\Psi_\delta(t)\,\rangle_\Omega\,dt + \int_0^T\int_\Omega f(t) \Psi_\delta(t)\,dx\,dt
\end{equation*}
are well-defined since $\Div \z(t) \in L^1(0,T;BV(\Omega)\cap L^\infty(\Omega))^*$,  $\Psi_\delta\in L^\infty(0,T;BV(\Omega)\cap L^\infty(\Omega))$ and $f\in L^1((0,T)\times\Omega)$. Hence the integral
\begin{equation}\label{tdz}
\int_0^T\ldoble\, \xi(t), \Psi_\delta(t)\rdoble_\Omega\, dt
\end{equation}
is also well-defined. To compute \eqref{tdz}, apply that $\xi$ is the time derivative of $u$ to deduce
\begin{align*}
\int_0^T\ldoble \xi(t), \Psi_\delta (t) \rdoble_\Omega\, dt&=-\frac1{\delta}
\int_0^T \int_\Omega \big(\eta(t)T_k(u(t))-\eta(t-\delta) T_k(u(t-\delta))\big)u(t) \, dx\, dt\\
&=-\frac1{\delta}\int_0^T \eta(t)\int_\Omega T_k(u(t)) \, u(t)\, dx \, dt+\frac1{\delta}\int_0^T \eta(t-\delta)\int_\Omega T_k(u(t-\delta))\, u(t)\, dx\, dt\\
&=\frac1{\delta} \int_0^T \eta(t) \int_\Omega T_k(u(t))\,\big(u(t+\delta)-u(t)\big)\, dx\, dt\,.
\end{align*}
The convexity of the real function $J_k$ implies
\begin{equation*}
T_k(u(t))\big(u(t+\delta)-u(t)\big)\le J_k(u(t+\delta))-J_k(u(t))\,.
\end{equation*}
Hence
\begin{multline*}
\int_0^T\ldoble \xi(t), \Psi_\delta (t) \rdoble_\Omega\, dt\le \frac1{\delta} \int_0^T\int_\Omega \eta(t)  \Big(J_k(u(t+\delta))-J_k(u(t))\Big)\, dx\, dt\\
=\frac1{\delta} \int_0^T\int_\Omega  \Big( \eta(t-\delta)J_k(u(t))-\eta(t) J_k(u(t))\Big)\, dx\, dt=\frac1{\delta} \int_0^T\big( \eta(t-\delta)-\eta(t)\big)\int_\Omega  J_k(u(t))\, dx\, dt\,.
\end{multline*}
Next, we will compute \eqref{tdz} by means of Green's formula
\begin{multline}\label{ec:15.0}
-\frac1{\delta}\int_0^T\big(\eta(t-\delta)-\eta(t)\big)\int_\Omega  J_k(u(t))\, dx\, dt\le -\int_0^T\ldoble \,\xi(t), \Psi_\delta(t)\,\rdoble_\Omega\, dt\\
=- \int_0^T  \frac1{\delta}\int_{t-\delta}^t\eta(s)\Big[ \langle \Div \z(t), T_k(u(s))\rangle_\Omega+ \int_\Omega f(t)T_k(u(s))\, dx\Big]\, ds\, dt\\
=\int_0^T  \frac1{\delta}\int_{t-\delta}^t\eta(s) \Big[\int_\Omega (\z(t), DT_k(u(s)))-\int_{\partial\Omega}T_k(u(s))[\z(t),\nu]\, d\mathcal H^{N-1}-\int_\Omega T_k(u(s)) f(t)\, dx\Big]\, ds\, dt\\
\le  \int_0^T  \frac1{\delta}\int_{t-\delta}^t\eta(s) \Big[\int_\Omega |DT_k(u(s))|-\int_{\partial\Omega}T_k(u(s))[\z(t),\nu]\, d\mathcal H^{N-1}-\int_\Omega T_k(u(s)) f(t)\, dx\Big]\, ds\, dt\,.
\end{multline}
Having in mind \cite[Lemma A3 and Corollary A4]{LS}, we let $\delta$ go to 0 to obtain
\begin{multline}\label{ec:15.1}
\int_0^T \eta'(t)\int_\Omega J_k(u(t))\,dx\, dt\\
\le \int_0^T  \eta(t) \Big[\int_\Omega |DT_k(u(t))|-\int_{\partial\Omega}T_k(u(t))[\z(t),\nu]\, d\mathcal H^{N-1}-\int_\Omega T_k(u(t)) f(t)\, dx\Big]\, dt\\
\le \int_0^T  \eta(t) \Big[\int_\Omega |DT_k(u(t))|+\int_{\partial\Omega}|T_k(u(t))|\, d\mathcal H^{N-1}-\int_\Omega T_k(u(t)) f(t)\, dx\Big]\, dt\,.
\end{multline}
On the other hand, taking $\eta(t)T_k(u_n(t))$ as test function in \eqref{cond_n0}, it yields
\begin{multline*}
-\int_0^T \eta'(t)\int_\Omega J_k(u_n(t))\, dx\, dt+\int_0^T \eta(t) \Big[\int_\Omega |DT_k(u_n(t))|+\int_{\partial\Omega}|T_k(u_n(t))|\, d\mathcal H^{N-1}\Big]\,dt\\
=\int_0^T\eta(t)\int_\Omega T_k(u_n(t))f_n(t)\, dx\, dt\,.
\end{multline*}
To pass to the limit in the first term we need that $J_k(u_n)\to J_k(u)$ strongly in $C([0,T];L^1(\Omega))$. This fact is a consequence of the inequality $|J_k(u_n)-J_k(u)|\le k|u_n-u|$ (which is easily derived from the Mean Value Theorem) on account of  $u_n\to u$ strongly in $C([0,T];L^1(\Omega))$. The second term may be handled by the lower semicontinuity of the BV-norm, while the right hand side is straightforward. Letting $n$ go to infinity, we get
\begin{multline*}
-\int_0^T \eta'(t)\int_\Omega J_k(u(t))\, dx\, dt+\int_0^T \eta(t) \left[\int_\Omega |DT_k(u(t))|+\int_{\partial\Omega}|T_k(u(t))|\, d\mathcal H^{N-1}\right]dt\\
\le \int_0^T\eta(t)\int_\Omega T_k(u(t)) f(t)\, dx\, dt\,.
\end{multline*}
Finally, having in mind \eqref{ec:15.1}, it implies
\begin{align}\label{ec:15.2}
\int_0^T \eta'(t)&\int_\Omega J_k(u(t))\,dx\, dt\\
\notag &\le \int_0^T  \eta(t) \Big[\int_\Omega |DT_k(u(t))|-\int_{\partial\Omega}T_k(u(t))[\z(t),\nu]\, d\mathcal H^{N-1}-\int_\Omega T_k(u(t)) f(t)\, dx\Big]\, dt\\
\notag &\le \int_0^T  \eta(t) \Big[\int_\Omega |DT_k(u(t))|+\int_{\partial\Omega}|T_k(u(t))|\, d\mathcal H^{N-1}-\int_\Omega T_k(u(t)) f(t)\, dx\Big]\, dt\\
\notag &\le\int_0^T \eta'(t)\int_\Omega J_k(u(t))\,dx\, dt\,.
\end{align}

It follows from \eqref{ec:15.2} that
\begin{equation*}
-\int_0^T  \eta(t) \int_{\partial\Omega}T_ku(t)[\z(t),\nu]\, d\mathcal H^{N-1}\, dt=\int_0^T  \eta(t) \int_{\partial\Omega}|T_ku(t)|\, d\mathcal H^{N-1}\, dt\,.
\end{equation*}
Since this identity holds for every $k>0$ and every non-negative $\eta\in C_0^\infty(\Omega)$, we get
\begin{equation*}
\int_{\partial\Omega}\big(|u(t)|+u(t)[\z(t),\nu]\big)\, d\mathcal H^{N-1}=0
\end{equation*}
for almost all $t\in(0,T)$, which implies the boundary condition (iii) (2).

Going back to \eqref{ec:15.2}, we also get
\[\int_0^T \eta'(t)\int_\Omega J_k(u(t))\,dx\, dt=\int_0^T  \eta(t) \Big[\int_\Omega |DT_ku(t)|+\int_{\partial\Omega}|T_ku(t)|\, d\mathcal H^{N-1}-\int_\Omega T_ku(t) f(t)\, dx\Big]\, dt\,.\]
The arbitrariness of $\eta$ leads to identity \eqref{u-test}.

\smallskip

{\bf Condition (iv).-}
Let $\phi\in L_w^1(0,T;BV(\Omega))\cap H^{1}(0,T;L^2(\Omega))\cap L^\infty((0,T)\times\Omega)$ be such that there exists $\z_\phi\in L^\infty((0,T)\times\Omega;\R^N)$ with $\|\z_\phi\|_\infty\le 1$ and $\Div\z_\phi\in L^2((0,T)\times\Omega)$. In addition, it holds $(\z_\phi(t),D\phi(t))=|D\phi(t)|$ and $[\z_\phi(t),\nu]\in\sg(-\phi(t))$ a.e. in $(0,T)$. Fixed $0<t<T$, thanks to Proposition \ref{ent-reg}, we obtain
\begin{multline}\label{sol-n}  
\left(\int_\Omega J_k(u_n(t)-\phi(t))\,dx\right)' + \int_\Omega \phi'(t)\,T_k(u_n(t)-\phi(t))\,dx \\
-\int_\Omega T_k(u_n(t)-\phi(t))\,\Div\z_\phi(t)\,dx \le \int_\Omega f_n(t)\,T_k(u_n(t)-\phi(t))\,dx\,,
\end{multline}
since $u_n$ is an entropy solution to problem \eqref{Pn-exist}.

It remains to pass to the limit as $n \to \infty$. To begin with, fix $0<s<t<T$ and integrate \eqref{sol-n} over $(s,t)$:
\begin{multline}\label{ec:exis}
\int_\Omega J_k(u_n(t)-\phi(t))\,dx-\int_\Omega J_k(u_n(s)-\phi(s))\,dx +\int_s^t \int_\Omega \phi'(\tau)\,T_k(u_n(\tau)-\phi(\tau))\,dx\, d\tau \\
-\int_s^t \int_\Omega T_k(u_n(\tau)-\phi(\tau))\,\Div\z_\phi(\tau)\,dx\, d\tau \le \int_s^t \int_\Omega f_n(\tau)\,T_k(u_n(\tau)-\phi(\tau))\,dx\, d\tau\,.
\end{multline}
Recalling \eqref{conv_4} and taking into account that
\begin{align*}
& \big|J_k(u_n(t)-\phi(t))\big|\le k\big(|u_n(t)|+|\phi(t)|\big)\,,\\
& \big|f_n(t)\,T_k(u_n(t)-\phi(t))\big|\le k|f_n(t)|\,,
\end{align*}
(notice that the right hand sides converge in $L^1((0,T)\times\Omega)$  to $k\big(|u(t)|+|\phi(t)|\big)$ and $ k|f(t)|$, respectively), every term of \eqref{ec:exis} converges and so we may take the limits to get
\begin{multline*}  
\int_\Omega J_k(u(t)-\phi(t))\,dx-\int_\Omega J_k(u(s)-\phi(s))\,dx +\int_s^t \int_\Omega \phi'(\tau)\,T_k(u(\tau)-\phi(\tau))\,dx\, d\tau \\
-\int_s^t \int_\Omega T_k(u(\tau)-\phi(\tau))\,\Div\z_\phi(\tau)\,dx\, d\tau \le \int_s^t \int_\Omega f(\tau)\,T_k(u(\tau)-\phi(\tau))\,dx\, d\tau\,.
\end{multline*}
Since this inequality holds for every $0<s<t<T$, we finally conclude that
\begin{multline*}
\left(\int_\Omega J_k(u(t)-\phi(t))\,dx\right)' + \int_\Omega \phi'(t)\,T_k(u(t)-\phi(t))\,dx \\
-\int_\Omega T_k(u(t)-\phi(t))\,\Div\z_\phi(t)\,dx \le \int_\Omega f(t)\,T_k(u(t)-\phi(t))\,dx
\end{multline*}
holds for almost all $t\in (0,T)$. 
Hence, the proof is complete.
\end{pf}

\begin{Remark}\label{igual-r}\rm
Having in mind \eqref{ec:15.2}, when taking the limit of \eqref{ec:15.0} as $\delta$ goes to 0 we obtain 
\begin{multline*}
\lim_{\delta\to0}\int_0^T  \frac1{\delta}\int_{t-\delta}^t\eta(s) \int_\Omega (\z(t), DT_ku(s))\, ds\, dt\\
=\lim_{\delta\to0}\int_0^T  \frac1{\delta}\int_{t-\delta}^t\eta(s) \int_\Omega |DT_ku(s)|\, ds\, dt=\int_0^T  \eta(t) \int_\Omega |DT_ku(t)|\, dt\,.
\end{multline*}
Appealing to \cite[Lemma A6]{LS}, the above limit leads to
\[\lim_{\delta\to0}\int_0^T\frac1\delta \int_{t-\delta}^t \left| \eta(s) \int_\Omega (\z(t), DT_ku(s))-\eta(t)\int_\Omega |DT_ku(t)|\right|\, ds\, dt =0\,.\]
Thus, in a {\it mean sense}, it leads to 
\[\int_0^T   \eta(t) \int_\Omega (\z(t), DT_ku(t))\, dt =\int_0^T \eta(t)\int_\Omega |DT_ku(t)|\, dt\]
for every non-negative $\eta\in C_0^\infty(0,T)$, so that we should infer that
\begin{equation}\label{igual}
\int_\Omega \big(\z(t), DT_k(u(t))\big)=\int_\Omega|DT_k(u(t))|
\end{equation}
holds for almost every $t\in (0,T)$. Nevertheless, we are not able to prove \eqref{igual} since the convergences involved are not sufficient to achieve this conclusion.

We point out that the identity \eqref{igual} is equivalent to 
\[\ldoble \xi(t), T_k(u(t)) \rdoble_\Omega=\left(\int_\Omega J_k(u(t))\, dx\right)'.\]
\end{Remark}

\section{Uniqueness}

In this section, we address the uniqueness of entropy solutions to problem \eqref{P} when the initial data and the source term are merely integrable functions. Specifically, we prove that the solution obtained as the limit of the approximating solutions in Theorem \ref{Teo-exist} is the unique entropy solution.

\begin{Theorem}\label{Teo-unic}
There exists a unique entropy solution to problem \eqref{P} when $f\in L^1((0,T)\times\Omega)$ and $u_0\in L^1(\Omega) $.
\end{Theorem}
\begin{pf} In order to show uniqueness, we follow the same pattern used in the $p$--Laplacian case (\cite{Pr} and \cite{AMST}).

Let $f\in L^1((0,T)\times\Omega)$ and $u_0\in L^1(\Omega)$, and denote by $u$ and $v$ two entropy solutions to \eqref{P} such that $u$ is obtained as a limit of the approximating problems \eqref{Pn-exist}. Hence, there exist sequences 
\begin{enumerate}
\item[(a)] $\{u_n\}$ in $C([0,T];L^2(\Omega))\cap L_w^1(0,T;BV(\Omega))\cap L^\infty((0,T)\times\Omega)\cap W^{1,1}(0,T;L^2(\Omega))$,
\item[(b)] $\{f_n\}$ in $L^\infty((0,T)\times\Omega)$,
\item[(c)] $\{u_n^0\}$ in $L^\infty(\Omega)\cap W^{1,1}(\Omega)$ such that $u_n(0)=u_n^0$,
\item[(d)] $\{\z_n\}$ in $ L^\infty((0,T)\times\Omega; \R^N)$ satisfying $\|\z_n\|_\infty\le 1$, $\Div\z_n \in L^2((0,T)\times\Omega)$ and conditions (i)--(iii) of Definition \ref{def-reg};
\end{enumerate}
satisfying:
\begin{enumerate}
\item[(e)] $u_n\to u$ in $L^\infty(0,T;L^1(\Omega))$,
\item[(f)] $f_n\to f $ in $L^1((0,T)\times\Omega)$,
\item[(g)] $u_n^0\to u_0 $ in $L^1(\Omega)$.
\end{enumerate}

Taking $\phi=u_n$ in the entropy formulation of $v$ and integrating between $s$ and $t$ (where $0<s<t<T$), it yields
\begin{multline}\label{des-1}
\int_\Omega J_k(v(t)-u_n(t))\,dx - \int_\Omega J_k(v(s)-u_n(s))\,dx + \ints u_n'(\sigma)T_k(v(\sigma)-u_n(\sigma))\,dx\,d\sigma \\
+ \ints \big(\z_{n}(\sigma),DT_k(v(\sigma)-u_n(\sigma))\big)\,d\sigma  - \int_s^t\intp T_k(v(\sigma)-u_n(\sigma))[\z_{n}(\sigma),\nu]\,d \h\,d\sigma  \\
\le \ints f(\sigma)T_k(v(\sigma)-u_n(\sigma))\,dx\,d\sigma\,.
\end{multline}
On the other hand, we choose the test function $T_k(v(\sigma)- u_n(\sigma))$ in problem \eqref{Pn-exist} (actually in \eqref{cond_n0}) and apply Green's formula to get
\begin{multline*}
\ints u_n'(\sigma)\,T_k(v(\sigma)- u_n(\sigma))\,dx\,d\sigma \\
= \ints T_k(v(\sigma)- u_n(\sigma))\Div \z_n(\sigma)\,dx\,d\sigma
+\ints f_n(\sigma)T_k(v(\sigma)- u_n(\sigma))\,dx\,d\sigma\\
 =- \ints \big(\z_n(\sigma),D T_k(v(\sigma)- u_n(\sigma))\big)\,d\sigma
 + \int_s^t\intp T_k(v(\sigma)-u_n(\sigma))[\z_n(\sigma),\nu]\,d\h\,d\sigma\\
+\ints f_n(\sigma)T_k(v(\sigma)- u_n(\sigma))\,dx\,d\sigma\,.
\end{multline*}
In other words:
\begin{multline}\label{des-2}
\ints u_n'(\sigma)\,T_k(v(\sigma)- u_n(\sigma))\,dx\,d\sigma
+\ints \big(\z_n(\sigma),D T_k(v(\sigma)- u_n(\sigma))\big)\,d \sigma\\
 - \int_s^t\intp T_k(v(\sigma)-u_n(\sigma))[\z_n(\sigma),\nu]\,d\h\,d\sigma
 =\ints f_n(\sigma)T_k(v(\sigma)- u_n(\sigma))\,dx\,d\sigma\,.
\end{multline}
Joining \eqref{des-1} and \eqref{des-2}, it implies
\begin{multline*}
\int_\Omega J_k(v(t)-u_n(t))\,dx-\int_\Omega J_k(v(s)-u_n(s))\,dx
\\
\le \ints (f(\sigma)-f_n(\sigma)) T_k(v(\sigma)- u_n(\sigma))\,dx\,d\sigma \,.
\end{multline*}
Taking into account the convergences satisfied by $u_n$ and $f_n$, it is straightforward to let $n$ go to infinity applying Lebesgue's Theorem and  so to arrive at 
\begin{equation*}
\int_\Omega J_k(v(t)-u(t))\,dx-\int_\Omega J_k(v(s)-u(s))\,dx\le0\,.
\end{equation*}
Dividing by $k>0$ and taking the limit as $k$ goes to $0$ it leads to
\begin{equation*}
\int_\Omega |v(t)-u(t)|\,dx\le \int_\Omega |v(s)-u(s)|\,dx \quad \text{ for } 0<s<t\,,
\end{equation*}
which implies that function $t\mapsto \displaystyle\int_\Omega |v(t)-u(t)|\,dx$ is non-increasing. Thus, for all $t\in (0,T)$, it implies
\begin{equation*}
\int_\Omega |v(t)-u(t)|\,dx\le \int_\Omega |v(0)-u(0)|\,dx =0\,,
\end{equation*}
from where we deduce the uniqueness of the solution.
\end{pf}

Since the existence and uniqueness of an entropy solution to problem \eqref{P} has been established for every $T>0$, the existence of a unique global entropy solution to problem \eqref{Pg} follows.

\section{Comparison between solutions}

Having established the well-posedness of problem \eqref{P}, this section is devoted to studying the stability of the solutions. We prove that the distance between two solutions in $C([0,T];L^1(\Omega))$ is controlled by the distance between their respective data. Additionally, we establish a comparison principle.
\begin{Theorem}
Let $u_1 $ and $u_2$ be the entropy solutions corresponding to data $ f_1,f_2 \in L^1((0,T)\times\Omega)$ and $u_0^1, u_0^2\in L^1(\Omega) $, respectively. It holds:
\begin{itemize}
\item[a)] If $u_0^1\le u_0^2 $ and $f_1 \le f_2 $, then $u_1 \le u_2 $.
\item[b)] $ \displaystyle \max_{t\in(0,T)}\|u_1(t)-u_2(t)\|_{L^1(\Omega)} \le \intT |f_1(t)-f_2(t)|\,dx\,dt + \int_\Omega |u_0^1-u_0^2|\,dx$.
\end{itemize}
\end{Theorem}
\begin{pf}
a) $\quad$
Let $u_i$ be an entropy solution to problem \eqref{P} with data $f_i\in L^1((0,T)\times\Omega)$ and $u_i^0\in L^1(\Omega)$, $i=1,2$. 
Now, consider the sequences $\{f_n^i\}\subset L^\infty((0,T)\times\Omega)$ and $\{u_n^{0,i}\}\subset L^\infty(\Omega)\cap H^1(\Omega)$ with
\begin{equation}\label{conv_comparacion1}
f_n^i\longrightarrow f_i \quad\text{in } L^1((0,T)\times\Omega) \qquad\text{ and }\qquad u_n^{0,i}\longrightarrow u_0^i \quad\text{in } L^1(\Omega)\,,
\end{equation}
such that $f_n^1\le f_n^2$ and $u_n^{0,1} \le u_n^{0,2}$ for every $n\in\N$. Moreover, due to Theorem \ref{Teo-exist}, we also know that 
\begin{equation}\label{conv_comparacion2}
u_n^i \longrightarrow u_i \quad\text{in } L^1((0,T)\times\Omega).
\end{equation}

Using the test function $T_k((u_n^1(t)-u_n^2(t))^+)$ in the approximate problem \eqref{Pn-exist} corresponding to data $f_n^1$ and $u_n^{0,1}$ and also in \eqref{Pn-exist} with data $f_n^2$ and $u_n^{0,2}$, and applying Proposition \ref{Prop.Generalizada}, we obtain
\begin{multline*}
\left(\int_\Omega J_k\big( (u_n^1(t)-u_n^2(t))^+\big)\,dx \right)'=\int_\Omega (u_n^1(t)-u_n^2(t))'\,T_k((u_n^1(t)-u_n^2(t))^+)\,dx  \\=
-\int_\Omega \big(\z_n^1(t)-\z_n^2(t),DT_k((u_n^1(t)-u_n^2(t))^+)\big) + \intp T_k((u_n^1(t)-u_n^2(t))^+)\big([\z_n^1(t),\nu]-[\z_n^2(t),\nu]
\big)\,d\h \\
+ \int_\Omega (f_n^1(t)-f_n^2(t))T_k((u_n^1(t)-u_n^2(t))^+)\,dx
\end{multline*}
which, dropping the non-positive terms and integrating over $(0,t)$, becomes
\begin{multline*}
\int_\Omega J_k\big((u_n^1(t)-u_n^2(t))^+\big)\,dx\\
 \le \int_\Omega J_k\big((u_n^1(0)-u_n^2(0))^+\big)\,dx + \intt (f_n^1(s)-f_n^2(s))T_k((u_n^1(t)-u_n^2(t))^+)\,dx\,ds 
\le 0\,.
\end{multline*}
Dividing by $k$ and letting $k\to0$, we get
\begin{equation*}
\int_\Omega (u_n^1(t)-u_n^2(t))^+ \,dx  \le 0 \,.
\end{equation*}
Using the convergences in $L^1((0,T)\times\Omega)$ of the solutions $u_n^1$ and $u_n^2$ (see \eqref{conv_comparacion2}) we get 
\begin{equation*}
\int_\Omega (u_1(t)-u_2(t))^+\,dx\le 0
\end{equation*}
which implies that $(u_1(t)-u_2(t))^+=0$, and so $u_1(t) \le u_2(t)$ for almost every $t\in(0,T)$.

\bigskip
b) $\quad$
The second part of the theorem is established by choosing the test function $T_k(u_n^1(t)-u_n^2(t))$ in problems with data $f_n^i$ and $u_n^{0,i}$, $i=1,2$. Making the same computations as in part a), we obtain
\begin{equation*}
\int_\Omega J_k(u_n^1(t)-u_n^2(t))\,dx
\le k \left\{\int_\Omega |u_n^{0,1}-u_n^{0,2}|\,dx+\intT  |f_n^1(s)-f_n^2(s)|\,dx\,ds\right\} \,.
\end{equation*}
Dividing by $k$ and passing to the limit as $k\to0$, we deduce
\begin{equation*}
\int_\Omega |u_n^1(t)-u_n^2(t)|\,dx \le  \int_\Omega |u_n^{0,1}-u_n^{0,2}|\,dx +\intT  |f_n^1(s)-f_n^2(s)|\,dx\,ds \,,
\end{equation*}
from where the result is inferred thanks to convergences \eqref{conv_comparacion1} and \eqref{conv_comparacion2}.
\end{pf}

As a consequence of the previous result, we obtain an inequality that provides a control on the solution $u$ in terms of its $L^1$-norm.
\begin{Corollary}
Let $u $ be the entropy solution to problem \eqref{P} corresponding to data $f\in L^1((0,T)\times\Omega)$ and $u_0\in L^1(\Omega)$. Then,  the following estimate holds:
\begin{equation*}
\max_{t\in(0,T)}\|u(t)\|_{L^1(\Omega)} \le \|u_0\|_{L^1(\Omega)}+\int_0^T \|f(t)\|_{L^1(\Omega)}\,dt \,.
\end{equation*}
\end{Corollary}

\section{Regularity}

This section establishes $L^r$-regularity for the entropy solutions. We demonstrate that if $f\in L^1(0,T;L^r(\Omega))$ and $u_0\in L^r(\Omega)$ for some $1<r<2$, then the solution not only belongs to $C([0,T];L^1(\Omega))$,  but also lies in $C([0,T];L^r(\Omega))$. This result is consistent with the regularity obtained in \cite{LS} (square-integrable data yields square-integrable solutions) and the existence result of the present work (merely summable data lead to merely summable solutions).

\begin{Theorem}\label{Teo-reg}
Let $1<r<2$. Assume  that  $u_0\in L^r(\Omega)$ and $f\in L^1(0,T;L^r(\Omega))$. Then, the unique entropy solution $u$ to problem \eqref{P} belongs to $C([0,T];L^r(\Omega))$.
\end{Theorem}
\begin{pf}
Fix $1<r<2$ and functions $u_0\in L^r(\Omega)$ and $f\in L^1(0,T;L^r(\Omega))$. Due to Theorems \ref{Teo-exist} and \ref{Teo-unic}, there exists a unique entropy solution to problem \eqref{P}. Moreover, this solution is obtained through an approximation procedure. We will show that such a solution also belongs to $C(0,T;L^r(\Omega))$.

Choosing sequences $\lbrace f_n\rbrace\subset L^\infty((0,T)\times\Omega)$ and $\lbrace u_n^0\rbrace\subset L^\infty(\Omega)\cap W^{1,1}(\Omega)$ such that $f_n\to f$ in $L^1(0,T;L^r(\Omega))$ and $u_n^0\to u_0$ in $L^r(\Omega)$, Proposition \ref{test-func} yields a  unique solution $u_n\in C([0,T];L^2(\Omega))\cap L^1_w(0,T;BV(\Omega))$ to every problem \eqref{Pn} with data $u_n^0$ and $f_n$. Let $\z_n$ denote the associated vector field.

Let us define now the auxiliary function $\omega(x,t):=\omega_{n,m}(x,t)=u_n(x,t)-u_m(x,t)$ and the bounded, non-decreasing, continuous function

\begin{equation*}
\varphi_k(s)=r|T_k(s)|^{r-1} \sg(s)\,.
\end{equation*}
Taking the test function $\varphi_k(\omega(s))$ in both problems (those corresponding to $m$ and $n$) and subtracting the resulting equations, we get
\begin{equation*}
\int_\Omega \varphi_k(\omega(s))\omega'(s)\,dx = \int_\Omega [\Div \z_n(s)-\Div \z_m(s)]\varphi_k(\omega(s))\,dx + \int_\Omega [f_n(s)-f_m(s)]\varphi_k(\omega(s))\,dx.
\end{equation*}
Defining $\Psi_k(t)=\int_0^t\varphi_k(s)\, ds$, using Proposition \ref{Prop.Generalizada} and dropping non positive terms, the previous equality becomes
\begin{equation*}
\left(\int_\Omega \Psi_k(\omega(s))\,dx\right)' = \int_\Omega \varphi_k(\omega(s))\omega'(s)\,dx \le \int_\Omega [f_n(s)-f_m(s)]\varphi_k(\omega(s))\,dx.
\end{equation*}
Now, we integrate over $(0,t)$ with $t\in(0,T)$:
\begin{equation}\label{ecu-1}
\int_\Omega \Psi_k(\omega(t))\,dx-\int_\Omega \Psi_k(\omega(0))\,dx 
\le \intt [f_n(s)-f_m(s)]\varphi_k(\omega(s))\,dx\,ds.
\end{equation}
The next step is to pass to the limit as $k\to+\infty$. On the one hand, by Young's inequality,
$$
|\Psi_k(\omega(s))| \le 2|\omega(s)|^r \le 2\left(\frac{r}{2}|\omega(s)|^2+\frac{2-r}{2}\right)\in L^1(\Omega)
$$
 hence, the Dominated Convergence Theorem leads to
\begin{equation*}
\lim_{k\to+\infty} \int_\Omega  \Psi_k(\omega(t))\,dx  =\int_\Omega|\omega(t)|^r\,dx = \| \omega(t)\|_r^r
\end{equation*}
and
\begin{equation*}
\lim_{k\to+\infty} \int_\Omega  \Psi_k(\omega(0))\,dx =\int_\Omega|\omega(0)|^r\,dx = \| \omega(0)\|_r^r.
\end{equation*}
On the other hand, $|\varphi_k(s)|=r|T_k(s)|^{r-1}\le r|s|^{r-1}\le r\left(\frac{r-1}{2}|s|^2+\frac{3-r}{2}\right)$ and so
\begin{equation*}
|[f_n(s)-f_m(s)]\varphi_k(\omega(s))| \in L^1((0,T)\times\Omega)
\end{equation*}
owing to $f_n, f_m\in L^\infty((0,T)\times\Omega)$.
Using again the Dominated Convergence Theorem, we get 
\begin{multline*}
\lim_{k\to+\infty} \intt [f_n(s)-f_m(s)]\varphi_k(\omega(s))\,dx\,ds =r \intt [f_n(s)-f_m(s)]|\omega(s)|^{r-1}\sg(\omega(s)) \,dx\,ds 
\\ \le r \int_0^t \left(\int_\Omega |f_n(s)-f_m(s)|^r\,dx\right)^{\frac{1}{r}}\left( \int_\Omega |\omega(s)|^r\,dx\right)^{\frac{r-1}{r}}\\
=r\int_0^t \|f_n(s)-f_m(s)\|_r\|\omega(s)\|_r^{r-1}\,ds
\end{multline*}
and inequality \eqref{ecu-1} becomes
\begin{equation*}
\|\omega(t)\|_r^r \le \|\omega(0)\|_r^r + r \int_0^t \|f_n(s)-f_m(s)\|_r\|\omega(s)\|_r^{r-1}\,ds\,.
\end{equation*}
Finally, applying a Gronwall type inequality due to Bihari \cite{Bi} (see \cite[Theorem 5.1]{BS}) we obtain
\begin{equation*}
\|\omega(t)\|_r \le \|\omega(0)\|_r + \|f_n-f_m\|_{L^1(0,T;L^r(\Omega))}\,;
\end{equation*}
and since $u_n^0\to u_0$ and $f_n\to f$ in $L^r(\Omega)$ and $L^1(0,T;L^r(\Omega))$, respectively, we conclude that $\{u_n\}$ is a Cauchy sequence in $L^\infty(0,T;L^r(\Omega))$. Since each $u_n\in C([0,T];L^2(\Omega))\subset C([0,T];L^r(\Omega))$, the result follows.
\end{pf}

\section{Long term decay for the homogeneous problem}

This final section analyzes the long-time behavior of solutions to the homogeneous problem. The qualitative properties and asymptotic behavior of the total variation flow have been thoroughly studied (see \cite{ACDM}) with bounded initial data; here we show that for initial data $u_0\in L^{r_0}(\Omega)$ with $r_0\in(1,2)$,  the solution  $u$ to problem \eqref{P} decays in norm over time in the $L^{r}(\Omega)$ spaces for every $r\in(1,r_0)$. This fact extends the main result of \cite{PR}.

\begin{Theorem}
Fix  $1< r_0<2$ and assume $u_0\in L^{r_0}(\Omega)$ and $f\equiv 0$. Then, for every $1< r<r_0$, the solution $u$ to problem \eqref{P} satisfies
\begin{equation*}
\|u(t)\|_{L^r(\Omega)}\le C\frac{\|u_0\|_{L^{r_0}(\Omega)}^{h_0}}{t^{h_1}} \quad\text{ for every }t>0.
\end{equation*}
with $h_0=\frac{r_0(N-r)}{r(N-r_0)}$, $h_1=\frac{N(r_0-r)}{r(N-r_0)}$ and $C=\left(\frac{N(r_0-r)}{N-r_0}\right)^{\frac{N(r_0-r)}{r(N-r_0)}}$.
\end{Theorem}
\begin{pf}
First, assume that $u$ is the solution to problem \eqref{P} with initial datum $u_0\in L^2(\Omega)\cap W^{1,1}(\Omega)$. In this first step, we argue as in \cite[Theorem 3.3]{PR}. We begin by applying  \cite[Theorem 4.1]{SW} which provides us with a unique solution  $u \in C([0,+\infty);L^2(\Omega))$. As mentioned, this solution is found as a limit of solutions to parabolic problems involving the $p$--Laplacian. More precisely: there exists a family $\{u_p\}_p$, where $1<p<2$, such that each $u_p$ 
is a solution to problem
\begin{equation}\label{asbe:0}
\left\{\begin{array}{rcll}
 u^\prime_p  - \Delta_pu_p & =&0\,,&\hbox{in } (0,+\infty)\times\Omega\,; \\[3mm]
 u_p &=& 0\,,  & \hbox{on } (0,+\infty)\times\partial \Omega\,;\\[3mm]
 u_p(x,0)&=&u_0(x)& \hbox{in }\Omega\,;
\end{array}\right.
\end{equation}
(so that $u_p\in  C([0,+\infty); L^2(\Omega))\subset C([0,+\infty); L^r(\Omega))$)
and this family satisfies, for every $t>0$,  
\[u_p(t)\to u(t)\qquad \hbox{strongly in } L^2(\Omega)\,.\]
It follows from \cite[Theorem 1.2]{Pz} that if $1<p<\frac{2N}{N+r_0}$ and $1<r<r_0$, then 
\begin{equation}\label{asbe:1}
\| u_p(t)\|_r\le C_p \frac{\|u_0\|_{r_0}^{h_{0,p}}}{t^{h_{1,p}}}
\end{equation}
where $h_{0,p}$, $h_{1,p}$ and $C_p$ depend only on the parameters and satisfy
\begin{equation*}
\lim_{p\to1}h_{0,p}=h_0:=\frac{r_0(N-r)}{r(N-r_0)}\,,\,\,\lim_{p\to1} h_{1,p}=h_1:=\frac{N(r_0-r)}{r(N-r_0)}\,,\,\,
\lim_{p\to1}C_p=C:=\left(\frac{N(r_0-r)}{N-r_0}\right)^{\frac{N(r_0-r)}{r(N-r_0)}}\,.
\end{equation*}
Therefore, letting $p$ tend to 1 in \eqref{asbe:1}, we get
\begin{equation}\label{asbe:2}
\| u(t)\|_r\le C \frac{\|u_0\|_{r_0}^{h_{0}}}{t^{h_{1}}}.
\end{equation}

For the general case, let $u_0\in L^{r_0}(\Omega)$. Then Theorem \ref{Teo-exist} and Theorem \ref{Teo-unic} guarantee that there exists a sequence of approximate solutions $\{u_n\}\subset C([0,+\infty);L^2(\Omega))$ such that
\[u_n^0\to u_0 \qquad \hbox{strongly in } L^r(\Omega)\,,\]
where $u_n^0\in L^\infty(\Omega)\cap H^1(\Omega)$ denotes the initial datum of $u_n$, and
\begin{equation*}
u_n(t)\to u(t)\qquad \hbox{strongly in }L^1(\Omega)\,,
\end{equation*}
for every $t>0$. 
Actually, we have deduced in the proof of Theorem \ref{Teo-reg} that, for every $t>0$,
\[u_n(t)\to u(t)\qquad \hbox{strongly in } L^{r}(\Omega)\,.\] 
On the other hand, since $u_n\in C([0,+\infty);L^2(\Omega))$, for every $t>0$, we obtain
\begin{equation}\label{asbe:3}
\| u_n(t)\|_r\le C \frac{\|u_n^0\|_{r_0}^{h_{0}}}{t^{h_{1}}}
\end{equation}
with the same constants as above.
Hence, passing to the limit as $n\to+\infty$ in \eqref{asbe:3}, the desired result is obtained.
\end{pf}

\section*{Funding}
The authors are partially supported by the grant PID2022-136589NB-I00 funded by MICIU/AEI/10.13039/501100011033 and by “ERDF A way of making Europe”.
M. Latorre also acknowledges partial support by Project PID2024-160967NB-I00 funded by AEI (Spain) and FEDER.
S. Segura de Le\'on also acknowledges partial support of Grant RED2022-134784-T funded by MCIN/AEI/10.13039/501100\-011033.

\end{document}